\documentclass[11pt]{amsart}
\usepackage{lmodern}
\usepackage[T1]{fontenc}
\usepackage{microtype}
\usepackage{amssymb}
\usepackage{amsthm}
\usepackage{amscd}
\usepackage[pdfstartview=FitH]{hyperref}
\usepackage{cite}

\usepackage{mathscinet}
\usepackage{latexsym}

\newtheorem{theorem}{Theorem}[section]
\newtheorem{lemma}[theorem]{Lemma}
\newtheorem{prop}[theorem]{Proposition}
\newtheorem{cor}[theorem]{Corollary}
\newtheorem{question}[theorem]{Question}

\theoremstyle{definition}
\newtheorem{definition}[theorem]{Definition}

\newtheorem{remark}[theorem]{Remark}

\newtheorem{example}[theorem]{Example}

\def \dom{\operatorname{dom}}

\makeatletter

\def\dotminussym#1#2{%
  \setbox0=\hbox{$\m@th#1-$}%
  \kern.5\wd0%
  \hbox to 0pt{\hss\hbox{$\m@th#1-$}\hss}%
  \raise.6\ht0\hbox to 0pt{\hss$\m@th#1.$\hss}%
  \kern.5\wd0}

\mathchardef\mhyphen="2D

\allowdisplaybreaks[2]

\newcommand{\RCA}{\ensuremath{\mathbf{RCA_0}}}
\newcommand{\WKL}{\ensuremath{\mathbf{WKL_0}}}

\newcommand{\RKL}{\ensuremath{\mathbf{RWKL}}}
\newcommand{\DNR}{\ensuremath{\mathbf{DNR}}}
\newcommand{\ADS}{\ensuremath{\mathbf{ADS}}}
\newcommand{\CAC}{\ensuremath{\mathbf{CAC}}}
\newcommand{\WWKL}{\ensuremath{\mathbf{WWKL}}}
\newcommand{\WRKL}{\ensuremath{\mathbf{RWWKL}}}
\newcommand{\EM}{\ensuremath{\mathbf{EM}}}
\newcommand{\SEM}{\ensuremath{\mathbf{SEM}}}
\newcommand{\SRT}{\ensuremath{\mathbf{SRT}^2_2}}

\newcommand{\K}{\mathcal{K}}
\newcommand{\D}{\mathcal{D}}
\newcommand{\R}{\mathcal{R}}
\newcommand{\W}{\mathcal{W}}
\usepackage{color}
\newcommand{\z}{z}

\newcommand{\dDNR}{d\mhyphen \mathrm{DNR}}
\newcommand{\upto}{\upharpoonright}
\newcommand{\comp}{\vcenter{\hbox{$\;\!\scriptstyle\circ\;\!$}}}

\begin{document}

\title{Separating Principles Below \WKL}
\author{Stephen Flood and Henry Towsner}
\date{\today}
\address {Department of Mathematics, University of Connecticut, Waterbury Campus, 99 East Main Street
Waterbury, CT 06702, USA}
\email{stephen.flood@uconn.edu}
\urladdr{\url{http://www.math.uconn.edu/~flood/}}
\thanks{Partially supported by NSF grant DMS-1340666.}
\address {Department of Mathematics, University of Pennsylvania, 209 South 33rd Street, Philadelphia, PA 19104-6395, USA}
\email{htowsner@math.upenn.edu}
\urladdr{\url{http://www.math.upenn.edu/~htowsner}}

\begin{abstract}
In this paper, we study Ramsey-type Konig's Lemma, written $\RKL$, using a technique introduced by Lerman, Solomon, and the second author.  This technique uses iterated forcing to construct an $\omega$-model satisfying one principle $T_1$ but not another $T_2$.  The technique often allows one to translate a ``one step'' construction (building an instance of $T_2$ along with a collection of solutions to each computable instance of $T_1$) into an $\omega$-model separation (building a computable instance of $T_2$ together with a Turing ideal where $T_1$ holds). 

We illustrate this translation by separating $d\mhyphen\DNR$ from $\DNR$ (reproving a result of Ambos-Spies, Kjos-Hanssen, Lempp, and Slaman), and then apply this technique to separate $\RKL$ from $\DNR$ (which has been shown separately by Bienvenu, Patey, and Shafer).
\end{abstract}

\maketitle

\section{Introduction}

Weak K\"onig's Lemma (\WKL), one of the ``Big 5'' systems of Reverse Mathematics, has developed a reputation for being a ``robust'' system---tweaking the statement tends to either leave its strength unchanged, or end up equivalent to one of a small handful of handful of other systems.  This behavior is in sharp contrast to the much-studied Ramsey's Theorem for Pairs, where each small variant seems to produce some different system.

A few distinct weakenings are known: the \textit{weak weak K\"onig's lemma} ($\mathbf{WWKL}$) is strictly weaker \cite{MR1080236}, and the existence of diagonally non-recursive functions ($\mathbf{DNR}$) is weaker still \cite{MR2135656}.  The first author introduced another principle, $\mathbf{RWKL}$ \cite{RKL}\footnote{This principle was originally named $\mathbf{RKL}$.  Here we follow the notation of \cite{RKL-Variants}.} which is strictly weaker that \WKL, and asked how it compares to $\mathbf{WWKL}$ and $\mathbf{DNR}$.  
In Section \ref{sec:rkl} we show that $\mathbf{RWKL}$ is strictly stronger than $\mathbf{DNR}$.  (This result has been separately shown by Bienvenu, Patey, and Shafer \cite{RKL-Variants} using very different methods; their method gives the stronger separation, that $\mathbf{WWKL}$ does not imply $\mathbf{RWKL}$.)

The method we use is based on a technique introduced by Lerman, Solomon, and the second author \cite{LST} which uses iterated forcing to construct an $\omega$-model satisfying one principle but not a second; we discuss this method in more detail below.  This method has a precursor in the literature---the separation of $\mathbf{WWKL}$ from $\mathbf{DNR}$ by Ambos-Spies, Kjos-Hanssen, Lempp, and Slaman \cite{MR2135656}.  In Section \ref{sec:wwkl} also give another proof of this result, using the same iterated forcing framework.

We also prove two related positive implications.  First, since $\mathbf{RWKL}$ and $\mathbf{WWKL}$ represent distinct weakenings of $\mathbf{WKL}$, it is natural to ask about combining them, into a ``$\mathbf{RWWKL}$''.  We show in Section \ref{sec:positive} principle is equivalent to $\mathbf{DNR}$.  (This was also shown by Bienvenu, Patey, and Shafer \cite{RKL-Variants} using other methods.)

The relationship between Ramsey's Theorem for Pairs and \WKL{} was open until Liu's recent proof \cite{MR2963024} that Ramsey's Theorem for Pairs does not imply \WKL{}  One of the motivations for the study of this problem is that all proofs of Ramsey's Theorem for Pairs appear to make use of \WKL{}, suggesting that Ramsey's Theorem for Pairs \emph{should} imply \WKL.  \RKL{} resolves this mystery: both \WKL{} and Ramsey's Theorem for Pairs imply \RKL{}, and \RKL{} suffices to carry out the proof of Ramsey's Theorem for pairs.  \EM{} is a weakening of Ramsey's Theorem for pairs, and its proof similarly makes use of \RKL{}; in Section \ref{sec:positive} we show that this is indeed necessary: \EM{} also implies \RKL{}.  (This has been independently shown by Bienvenu, Patey, and Shafer \cite{RKL-Variants}.)

\section{Principles and Definitions}

Throughout this paper our base theory is always the theory \RCA{} of Reverse Mathematics as described in \cite{simpson99}.

\begin{definition}
  We write $2^\omega$ for the collection of functions from $\mathbb{N}$ to $\{0,1\}$ and $2^{<\omega}$ for the set of functions from some initial segment $[0,n]$ to $\{0,1\}$.  If $\sigma\in 2^{<\omega}$ we write $|\sigma|$ as an abbreviation for $|\dom(\sigma)|$, the length of $\sigma$.

A \emph{tree} is a set $T\subseteq 2^{<\omega}$ such that for every $\sigma\in T$ and every $n<|\sigma|$, $\sigma\upharpoonright[0,n]\in T$; we write $[T]$ for the set of $\Lambda\in 2^\omega$ such that for every $n$, $\Lambda\upharpoonright n\in T$.  We write $T_n$ for the set of $\sigma\in T$ with $|\sigma|=n$ and $|T|=\max\{|\sigma|\mid \sigma\in T\}$ (and $|T|=\infty$ if $T$ contains infinitely many elements).
\end{definition}
The definition of $|T|$ will not cause confusion since we will never be interested in the cardinality of $T$.

\begin{definition}
  We write $\tau\sqsubseteq\sigma$ if $\tau=\sigma\upharpoonright[0,n]$ for some $n<|\sigma|$.

  We say $T'$ \emph{end-extends} $T$ if $T\subseteq T'$ and whenever $\sigma\in T'\setminus T$ there is a $\tau\sqsubset\sigma$ with $\tau\in T$ and $|\tau|=|T|$.
\end{definition}

\begin{definition}
  $\WKL$ states that whenever $T\subseteq 2^{<\omega}$ is infinite, $[T]$ is non-empty.  $\WWKL$ states that whenever $T\subseteq 2^{<\omega}$ and there is an $\epsilon$ such that for every $n$, $|T_n|\geq \epsilon 2^n$, $[T]$ is non-empty.

$\RKL$ states that whenever $T\subseteq 2^{<\omega}$ is infinite, there exists an infinite set $H$ so that for every $n$, there is a $\sigma\in T_n$ such that $\sigma\upharpoonright H$ is constant.  $\mathbf{RWWKL}$ states that whenever $T\subseteq 2^{<\omega}$ is infinite and there is an $\epsilon$ such that for every $n$, $|T_n|\geq \epsilon 2^n$, there exists an infinite set $H$ so that for every $n$, there is a $\sigma\in T_n$ such that $\sigma\upharpoonright H$ is constant. 
\end{definition}

\begin{definition}
  $\DNR$ states that for every set $X$ there exists a total function $f$ such for every $e$ such that $\phi^X_e(e)\downarrow$, $f(e)\neq \phi^X_e(e)$.

When $d$ is a total function, $d\mhyphen\DNR$ states that for every set $X$ there exists a total function $f$ such for every $e$ such that $\phi^X_e(e)\downarrow$, $f(e)\neq \phi^X_e(e)$, and $f(e)<d(e)$ for all $e$.
\end{definition}

It is well known that $2\mhyphen\DNR$ is equivalent to $\WKL$, and by \cite{MR1034562}, there is a computable $d$ such that $\WWKL$ implies $d\mhyphen\DNR$.

\section{\texorpdfstring{\DNR{}}{DNR} does not imply \texorpdfstring{$d\mhyphen\DNR$}{d-DNR}}\label{sec:wwkl}

In this section we give a proof that $\DNR$ does not imply $\WWKL$.  This was originally shown in \cite{MR2135656}, and our proof has the same underlying structure: we show that, for any computable $d$, $\DNR$ does not imply $d\mhyphen\DNR$.  We show this in an iterative construction, beginning with a well-chosen instance of $d\mhyphen\DNR$ and successively adding sets resolving instances of $\DNR$ in such a way that we never solve our chosen instance of $d\mhyphen\DNR$.  We include this proof here to illustrate the connection between the iterated forcing method of \cite{LST} and the proof in \cite{MR2135656}, and to introduce some of the ideas we use in the next section.

\subsection{Families of Extensions}

We will work with finite approximations to \DNR{} functions:
\begin{definition}
  We write $\omega^{\subset\omega}$ for the set of partial functions from $\omega$ to $\omega$ with finite domain.  We say $f\in\omega^{\subset\omega}$ is \emph{$\mathrm{DNR}^X$} if for every $e\in\dom(f)$ such that $\varphi^X_e(e)\downarrow$, $\varphi^X_e(e)\neq f(e)$.  If $d:\omega\rightarrow\omega$, we say $f$ is \emph{$d\mhyphen\mathrm{DNR}^X$} if $f$ is $\mathrm{DNR}^X$ and $f(x)<d(x)$ for all $x\in\dom(f)$.
  \end{definition}

When we want to extend a $\mathrm{DNR}^X$ partial function $f$ with finite domain, we generally want many extensions of $f$ to choose from.  Specifically, we generally want to fix some $x\not\in\dom(f)$ and have $n$ choices for the value of the extension at $x$.  Furthermore, for each such choice, we should have some $x'\not\in\dom(f)\cup\{x\}$ and $n$ choices for the value of the extension at $x'$, and so on.  We formalize this with the notion of an $n$-branching set of extensions of $f$ of length $k$.  The length is the number of new values which will be added to the partial function's domain.
The length is only specifically referenced in the recursive definition and a few easy lemmas which are shown by induction on the length.  
\begin{definition}
  Let $f\in\omega^{\subset\omega}$.  We define an \emph{$n$-branching set of extensions of $f$ of length $k$} by induction on $k$:
  \begin{itemize}
  \item An $n$-branching set of extensions of length $1$ is a set $U$ of $n$ functions so that for some fixed $x\not\in\dom(f)$, each $g\in U$ satisfies $f\subseteq g$ and $\dom(g)=\dom(f)\cup\{x\}$,
\item If $U_0$ is an $n$-branching set of extensions of $f$ of length $k$ and if $(\forall g\in U_0)[U_g$ is an $n$-branching set of extensions of $g$ of length $1]$, then $\bigcup_{g\in U_0}U_g$ is an $n$-branching set of extensions of $f$ of length $k+1$.
  \end{itemize}
\end{definition}

Our notion of an $n$-branching set of extensions is very similar to the notion of an $n/m$-good tree from \cite{MR2135656}.

\begin{lemma}\label{lemma.branch_gives_DNR}
  If $f$ is $\mathrm{DNR}^X$ and $U$ is a $2$-branching set of extensions of $f$ then there is a $g\in U$ which is $\mathrm{DNR}^X$.
\end{lemma}

Importantly, $n$-branching sets of extensions satisfy a pigeonhole principle:
\begin{lemma}\label{tree_pigeonhole}
  Suppose $U$ is an $n+m-1$-branching set of extensions of $f$ and let $R\subseteq U$.  Then either there is an $n$-branching set of extensions $U_R\subseteq R$ or there is an $m$-branching set of extensions $U_B\subseteq U\setminus R$.
\end{lemma}
\begin{proof}
  By induction on the length of $U$.  For convenience, let $B=U\setminus R$.  If $U$ has length $1$, this is immediate---either $|R|\geq n$, in which case any $n$ element subset of $R$ suffices, or $|R|<n$, so $|B|\geq m$, and any $m$ element subset of $B$ suffices.

  Suppose $U$ has length $k+1$.  Then there is some $U_0$ which is an $n+m-1$-branching set of extensions of $f$ of length $k$ and for each $g\in U_0$ there is some $U_g$ which is an $n+m-1$-branching set of extensions of $g$ of length $1$ such that $U=\bigcup_{g\in U_0}U_g$.  By the previous paragraph, for each $g\in U_0$ there is some $U_g'$ where either $|U'_g|=n$ and $U'_g\subseteq R$ or $|U'_g|=m$ and $U'_g\subseteq B$.  Let $\iota_g=R$ in the former case and $\iota_g=B$ in the latter case.  Define $R_0\subseteq U_0$ to be those $g\in U_0$ such that $\iota_g=R$.  By the inductive hypothesis, either we have an $n$-branching $U_R\subseteq R_0$, and so $\bigcup_{g\in U_R}U'_g$ suffices, or an $m$-branching $U_B\subseteq (U_0\setminus R_0)$, in which case $\bigcup_{g\in U_B}U'_g$ suffices.
\end{proof}

We also need the corresponding iterated version:
\begin{lemma}\label{lemma.dDNR-iteratedpigeon}
Suppose $U$ is a $k(n-1)+1$-branching set of extensions of $f$ and $U=\bigcup_{i<k}U_i$.  Then there is some $U_i$ which contains an $n$-branching set of extensions of $f$.
\end{lemma}

\subsection{Families of Completions}

\begin{definition}
  If $f\in\omega^{\subset\omega}$, a \emph{family of completions} of $f$ is a set $\Xi\subseteq\omega^{\subset\omega}$ such that:
  \begin{itemize}
  \item If $g\in\Xi$ then $f\subseteq g$,
  \item If $g\in\Xi$ and $f\subseteq h\subseteq g$ then $h\in\Xi$.
  \end{itemize}

We say $\Xi$ \emph{blocks at width $n_\Xi$} if every $n_\Xi$-branching set of extensions of $f$ has non-empty intersection with $\Xi$.
\end{definition}
A family of completions $\Xi$ naturally defines a (possibly empty) collection of total extensions of $f$, namely those total functions $f^\infty$ such that for all finite $g\subset f^\infty$, $g\in\Xi$.  Note that for any given $g\subset f^\infty$, it is likely that $g$ is \emph{not} an initial segment of $f$.

When the family is blocking, we know that this collection is actually non-empty.  Note that there may still be branches in $\Xi$ which are dead ends--- which have no further extensions---but we will see in Lemma \ref{thm:dnr_extension} that blocking families also contain many elements which are not dead ends. 

\begin{definition}
  If $\Xi$ is a family of completions of $f$ and $g\in\Xi$, $\Xi\upharpoonright g$ is the set of $h\in\Xi$ such that $g\subseteq h$.
\end{definition}

\begin{lemma}\label{thm:dnr_extension}
Let $\Xi$ be a family of completions of $f$ which blocks at width $n_\Xi$ and let $U$ be a $2n_\Xi$-branching set of extensions of $f$.  Then there is a $2$-branching set $U^*\subseteq U\cap \Xi$ such that for every $g\in U^*$, $\Xi\upharpoonright g$ blocks at width $n_\Xi$.
\end{lemma}
\begin{proof}
  By Lemma \ref{tree_pigeonhole}, either there is an $n_\Xi+1$-branching set of extensions of $f$, $U'\subseteq U\cap\Xi$, or there is an $n_\Xi$-branching set of extensions $U'\subseteq U$ with $U'\cap\Xi=\emptyset$.  Since the latter violates the assumption that $\Xi$ blocks at width $n_\Xi$, we have a $n_\Xi+1$-branching set of extensions of $f$, $U'\subseteq U\cap\Xi$.

Consider those $g\in U'$ such that there is a $n_\Xi$-branching set of extensions of $g$, $U_g$, disjoint from $\Xi$.  These $g$ cannot contain an $n_\Xi$-branching subset so by the same argument as above, we have $U^*\subseteq U'$ which is $2$-branching and such that whenever $g\in U^*$, $\Xi\upharpoonright g$ blocks at width $n_\Xi$.
\end{proof}

\begin{example}
Suppose that $\Xi$ is a family of completions blocking at width $3$. 
By definition, if any $U$ is at least $3$-branching, then \emph{one} element of $U$ is {in} $\Xi$.
Lemma \ref{thm:dnr_extension} shows that if $U$ is at least $6$-branching, then there are \emph{two} elements $g_0,g_1\in U$ in $\Xi$ with the stronger property that each $\Xi\upharpoonright g_i$ is also blocking at width $3$.
\end{example}

\begin{remark}
Note that because each 6-branching set is automatically 3-branching, the property of $U$ being $n$-branching becomes more restrictive as $n$ grows.
Therefore, the property of $\Xi$ blocking at width $n$ becomes less restrictive as $n$ grows.   
\end{remark}

For example, $\Xi=2^{<\omega}$ is a family of completions of $\lambda$ blocking at width $1$.
During the construction, we will be able to thin out $\Xi$ in exchange for permitting it to block at larger $n_\Xi$.

\subsection{The one step case}\label{sec:onestep}
The full proof requires an iterated forcing argument, but the main idea is coveyed in the ``one-step'' case.
In other words, we will solve only a single instance of $\mathrm{DNR}$ rather than iteratively creating a model of $\DNR$.
In sections \ref{sec:dDNR-iteration} and \ref{sec:dDNR-Ground}, this is extended to a full separation of $\DNR$ and $\WWKL$.

Let $d$ be a computable function.  We wish to construct a set $V^\infty$ and a $\mathrm{DNR}^{V^\infty}$ function $f^\infty$ so that no function computable from $V^\infty\oplus f^\infty$ is $d\mhyphen\mathrm{DNR}^{V^\infty}$.  (The existence of such $V^\infty$ and $f^\infty$, with $V^\infty$ computable, is Theorem 2.1 of \cite{MR2135656}, based on essentially the same result from \cite{MR2559129}; the argument here is based on theirs.) 

As usual, a solution to $\dDNR^{V^\infty}$ is a $ d$-bounded total function $g$ such that for each $e$, if $\Phi^{V^\infty}_e(e)\downarrow$ then $g(e)\neq \Phi^{V^\infty}_e(e)$. 
This definition is not ideal for the the forcing construction.
To simplify matters, we will work with another sense of ``diagonalization'' which is a consequence of $\dDNR^{V^\infty}$. 

\begin{definition}\label{defn.diagonalizes-against}
  Let $V^\infty$ be a partial function such that whenever $V^\infty(x)$ is defined, $V^\infty(x)=(i,n)$ with $i<d(x)$; for each $x$ such that $V^\infty(x)=(i,n)$ is defined, we set $V_0^\infty(x)=i$.  We say a total function $r$ \emph{diagonalizes against} $V^\infty$ if for every $x$, $r(x)<d(x)$ and whenever $V^\infty_0(x)$ is defined, $r(x)\neq V^\infty_0(x)$.
\end{definition}
\begin{prop}
There is a computable function $s$ so that any $d\mhyphen\mathrm{DNR}^{V^\infty}$ function computes a $d\comp s$-bounded total function $r$ which diagonalizes against $V^\infty$.
\end{prop}
\begin{proof}
By the s-m-n theorem, there is a total computable function $s$ such that for each $y$ and each $x$, $\Phi^{V^{\infty}}_{s(x)}(y) = i$ if there is some $n$ s.t. $V^\infty(x)=(i,n)$.
Note that $V_0^\infty(x)= \Phi^{V^{\infty}}_{s(x)}(y)$ for every input $y$.

Let $g$ be any $\dDNR^{V^\infty}$ function.  Then $g(s(x))\neq \Phi^{V^{\infty}}_{s(x)}(s(x))$ whenever the computation halts.
Finally, define $r(x) = g(s(x))$ for each $x$.
Then whenever $V_0^\infty(x)= \Phi^{V^{\infty}}_{s(x)}(s(x))$ is defined, $r(x)\neq V_0^\infty(x)$ as desired.
\end{proof}

\begin{definition}
  We write $\hat d(x)=d(s(x))$.
\end{definition}

Therefore it suffices to construct $V^\infty$ and a $\mathrm{DNR}^{V^\infty}$ function $f^\infty$ so that no function computable from $V^\infty\oplus f^\infty$ diagonalizes against $V^\infty$.

\begin{remark}
The second coordinate $V_1^\infty$ will be used in the forcing construction to allow us to define $V_0^\infty(x)$ at any stage without injuring any computations previously performed with oracle $V^\infty$. 
\end{remark}

To build these sets by forcing, our conditions will need to record some additional information.  
To build $V^\infty$, we will use a number $m$ to record the use of the longest computation used so far, and a finite set $R$ to record the inputs where we promise that $V^\infty$ will be undefined.  

To build $f^\infty$, we will use a function $\Xi$ from finite sets $V'$ to a family of completions $\Xi(V')$ that represents the possible options for $g\succeq f$ in the case we extend $V$ to $V'$. 
The ability to define different sets of completions for different future $V'$ will be used in the case where we force $\Phi_e^{V''\oplus g}(x)$ to diverge or be greater than $\hat d(x)$ for all future $V''$ and $g$. 

When building $f^\infty$, we will also use a number $n_\Xi$ to record the fixed amount of blocking satisfied by every element of the range of $\Xi$.  
As the construction proceeds $n_\Xi$ may increase (e.g.\ when we force $\Phi_e^{V^{\infty}\oplus f^\infty}$ to diverge or exceed $\hat d(x)$), but it will always remain a finite number.  

Formally, we will construct the desired pair $V^\infty,f^\infty$ by forcing with tuples
\[(m,V,R,f,n_\Xi,\Xi).\]
 where:
\begin{itemize}
\item $V$ is a finite partial function on $[0,m]$ such that when $x\in\dom(V)$, $V(x)=(i,n)$ for some $i<\hat d(x)$ and $n\leq m$,
\item $R$ is a finite set with $R\cap\dom(V)=\emptyset$,
\item $f\in\omega^{\subset\omega}$ is $\mathrm{DNR}^{V}$,
\item $\Xi$ is a function so that 
\begin{enumerate}
	\item when $V'\supseteq V$ with $\dom(V')\cap R=\emptyset$, $\Xi(V')$ is a family of completions of $f$ blocking at width $n_\Xi$, and
	\item If $V\subseteq V'\subseteq V''$ then $\Xi(V'')\subseteq \Xi(V')$.
\end{enumerate}

\end{itemize}

When we formalize this more carefully for the full construction, we will need $\Xi(V)$ to be precisely $\{f\mid \forall n\ C(n,f,V)\}$ for some computable relation $C$.  When formalized this way, being a condition is a $\Pi^0_1$ statment.  Therefore, because failing to be a condition is a $\Sigma^0_1$ question, we can assume something is a condition until we find a witness that it isn't.  
We will take this further in the iterated forcing, where we will force statements to hold for all conditions by considering each ``pre-condition'' and either demonstrating that the statement holds, or forcing a witness that the pre-condition is not a condition.

We say $(m',V',R',f',n'_\Xi,\Xi')\preceq(m,V,R,f,n_\Xi,\Xi)$ if 
\begin{itemize}
\item $V\subseteq V'$ and $R\subseteq R'$,
\item $m\leq m'$, and if $x\in\dom(V')\setminus \dom(V)$ with $V'(x)=(i,n)$ then $n>m$,
\item $f\subseteq f'$ and $f'\in\Xi(V')$, and
\item For every $V''\supseteq V'$, $\Xi'(V'')\subseteq\Xi(V'')\upharpoonright f'$.
\end{itemize}

Note that we can avoid the apparent $\Pi^0_2$ character of the last clause by only considering extensions which witness the extension syntactically (that is, in an immediate extension, $\Xi'(V'')$ should have the form
\[\{f\mid \forall n\ C'(n,f,V'')\wedge C(n,f,V'')\wedge C_0(n,f,V'')\wedge\cdots\wedge C_k(n,f,V'')\}\]
where $C_0$ through $C_k$ are the relations corresponding to previous conditions).

Note that $V'$ and $R'$ both potentially grow.  $V$ is an encoding of a partial function, while $R$ is the set of places where $V$ is undefined.  As we force, we extend both the places where $V$ is defined and also the places where $V$ is forced to be undefined.

Given $V$, we write $V_0,V_1$ for the functions with $\dom(V_0)=\dom(V_1)=\dom(V)$ so that $V(x)=(V_0(x),V_1(x))$ for all $x\in\dom(V)$.  
Because the definition of extending conditions requires that $V_1(x)>m$, extending to a condition with $V_0(x)=i$ will not not injure any previously referenced computation (which will all have use at most $m$).

It is convenient to write $V'\preceq(m,V,R)$ if $V\subseteq V'$, $\dom(V')\cap R=\emptyset$, and for each $x\in\dom(V')\setminus\dom(V)$, $V'_1(x)>m$.
Note that because $V'(x) > V'_1(x) > m$ for all $x\in \dom(V')\setminus \dom(V)$, and because $m$ is the use of the longest computations performed so far, then extending to $V'$ does not injure any computations which converged with oracle $V$.

We will construct an infinite sequence $(m^0,V^0,R^0,f^0,n^0_\Xi,\Xi^0)\succeq\cdots$, and we will let $V^\infty=\bigcup V^n$ and $f^\infty=\bigcup f^n$.  We speak of a condition $(m,V,R,f,n_\Xi,\Xi)$ forcing some statement regarding $V^\infty$ and $f^\infty$ if the statement will be true of every $V^\infty,f^\infty$ coming from such a sequence which includes $(m,V,R,f,n_\Xi,\Xi)$.

\bigskip
The main requirement we must satisfy is that for each $e$, we should force that $\Phi^{V^\infty\oplus f^\infty}_e$ is not a total $\hat d$-bounded function which diagonalizes against $V^\infty$.  

\begin{lemma}
\label{lemma.dDNR-one-step}
Given any $(m,V,R,f,n_\Xi,\Xi)$ and any $e\in\mathbb{N}$, there is an $x$ and a 
$(m',V',R',f',n'_\Xi,\Xi')\preceq(m,V,R,f,n_\Xi,\Xi)$ such that either:
\begin{itemize}
\item $\Phi^{V'\oplus f'}_e(x)=V'_0(x)$, or
\item Whenever $V''\preceq(m',V',R')$ and $g\in\Xi'(V'')$, if $\Phi^{V''\oplus g}_e(x)\downarrow$ then $\Phi^{V''\oplus g}_e(x)\geq \hat d(x)$.
\end{itemize}
\end{lemma}
\begin{proof}
Let $(m,V,R,f,n_\Xi,\Xi)$ be given; choose a value $x\not\in R\cup\dom(V)$.  We ask whether the following exists:
\begin{quote}
A $V'\preceq(m,V,R\cup\{x\})$ and a $\hat d(x)(2n_\Xi-1)+1$-branching set of extensions $U$ of $f$ so that for every $g\in U$, $\Phi^{V'\oplus g}_e(x)\downarrow$ and $\Phi^{V'\oplus g}_e(x)<\hat d(x)$.
\end{quote}

Suppose not.  
Then we can restrict to a family of completions $\Xi'$ which will allow us to force the second case.
For each $V'\preceq(m,V,R\cup\{x\})$, we define 
$\Xi'(V')=\{g\in\Xi(V')\mid\Phi^{V'\oplus g}_e(x)\uparrow\text{ or }\Phi^{V'\oplus g}_e(x)\geq \hat d(x)\}$.
By assumption,  every $\hat d(x)(2n_\Xi-1)+1$-branching set of extensions $U$ of $f$ contains an element of $\Xi'(V')$, so $\Xi'(V')$ satisfies the definition of blocking at width $\hat d(x)(2n_\Xi-1)+1$.  We now extend to the condition
\[(x,V,R\cup\{x\},f,\hat d(x)(2n_\Xi-1)+1,\Xi')\preceq(m,V,R,f,n_\Xi,\Xi).\]
Note that for any $V''\preceq(x,V,R\cup\{x\})$ and any $g\in\Xi'(V'')$, either $\Phi^{V'\oplus g}_e(x)\uparrow$ or $\Phi^{V'\oplus g}_e(x)\geq \hat d(x)$, so the second case is forced.
 
Suppose instead that we find such a $V'$ and such a $U$.  We color $g\in U$ by $i<\hat d(x)$ based on the value of $\Phi^{V'\oplus g}_e(x)$.  There is a $2n_\Xi$-branching $U'\subseteq U$ and an $i<\hat d(x)$ so that for each $g\in U'$, $\Phi^{V'\oplus g}_e(x)=i$.  Now we consider the extension $V''=V'\cup\{(x,(i,|V'|+1))\}$ and let $m'$ be larger than any value in $V''$, so $(m',V'',R)$ is a valid partial condition.

All that remains is picking a particular $g\in U'$ to extend $f$ to.  This requires a bit of care, because it could be that there is some $g\in U'\cap \Xi(V'')$, but for some $V^*\preceq(m',V'',R)$, $g\not\in\Xi(V^*)$.\footnote{For example, if we forced divergence at a previous stage, strings might leave $\Xi(V'')$ once $V''$ is long enough to cause the relevant computation to converge.}  Suppose there were such a $g$; then we could replace $V''$ with this $V^*$, changing $m'$ accordingly.  Each time we do this we cause $|U'\cap \Xi(V'')|$ to decrease.  However, the level of blocking of $\Xi$ and the branching of $U'$ remains constant, so we do this only finitely many times.

Therefore, without loss of generality, we may assume that this never happens---that for every $V^*\preceq(m',V'',R)$, $U\cap\Xi(V'')=U\cap\Xi(V^*)$.

Similarly, we know from Lemma \ref{thm:dnr_extension} that for each $V^*\preceq(m',V'',R)$ there is some $g\in U'$ such that $\Xi\upharpoonright g$ blocks at width $n_\Xi$.  Again, however, the choice of $g$ could depend upon $V^*$.  We may use the same process: if there is a $g\in U'$ and such a $V^*$, we replace $V''$ with $V^*$; each time we do this, the set of $g$ blocking at width $n_\Xi$ decreases, so without loss of generality, we may assume that for every $V^*\preceq(m',V'',R)$, if $\Xi(V'')\upharpoonright g$ blocks at width $n_\Xi$ then $\Xi(V^*)$ blocks at width $n_\Xi$ as well.

By Lemma \ref{thm:dnr_extension}, we choose a $g\in U'\cap\Xi(V'')$ such that $\Xi(V'')\upharpoonright g$ blocks at width $n_\Xi$ and pass to the condition $(m',V'',R,g,n_\Xi,\Xi\upharpoonright g)$; our choice of $V''$ and $g$ ensures that this is a requirement.  We have $\Phi^{V''\oplus g}_e(x)=i=V''_0(x)$, so we have forced that $\Phi^{V^\infty\oplus g^\infty}_e(x)$ does not diagonalize against $V^\infty$.
\end{proof}

We will next adapt this ``one step'' construction to produce an $\omega$-model separation.  Towards this end, we will separate Lemma \ref{lemma.dDNR-one-step} into two parts: one concerned with extending the function $f$ (the iterated forcing), and one concerned with extending the set $V$ (the ground forcing).

\begin{remark}[Ground Forcing]
In both the one-step forcing and the ground forcing, we ask ``what can I force the strings $f$ to do as oracles'' and then we define $V_0(x)$ accordingly.
Because of this similarity, the first step of Lemma \ref{lemma.dDNR-one-step} where we break into cases and find $V'$ and $U$ is essentially repeated in the ground-forcing Lemma \ref{lemma.dDNR-ground-iteration}.  
The two main changes are that (1) the particular requirement concerning $\dDNR^V$  is replaced with the more general requirement $\mathcal{K}$, and (2) we will merely be guaranteeing density of $\mathcal{K}$ rather than actually selecting $g$.  
\end{remark}

\begin{remark}[Avoiding Dead Ends when Forcing $g$]
Note that before we chose which $g\in U'$ that we extended $f$ to, we needed to avoid potential dead ends.
This means that whenever it was possible, we extended $V$ to $V^*$ so that $(g,n_\Xi,\Xi)$ was not a valid condition.
We were able to use the constant amount of blocking in $\Xi$ to verify that this process had to stop: we could not prevent \emph{all} the $g$ from being valid conditions.
This idea is incorporated into the Ground Forcing Lemma \ref{lemma.dDNR-ground-iteration}.
\end{remark}

In the iteration forcing, we are given a fixed $V_0$, and we `hope' that we have enough density of strings in $\Xi$ to permit us to force $f$ to do the right thing as an oracle.  
Thus, although ideas and themes from Lemma \ref{lemma.dDNR-one-step} do appear in the iteration forcing Lemma \ref{lemma-dDNR-iteration-extending-f}, there is a much weaker resemblance.

\begin{remark}[Iteration Forcing]
In the iteration forcing, we will introduce generalized ``requirements'' will consist of a family of strings that have/give the ``correct'' response to the given values of $V_0(x)$. 
We will say that a requirement is ``uniformly dense'' if it can either be strongly avoided (we will call these requirements ``not essential'') or if it is sufficiently blocking so that there are enough strings in $\Xi$ so that we can satisfy the requirement and continue the construction.

The main work of the iterated construction is to determine the amount of density required for us to satisfy conditions $\mathcal{K}^X$ and to show that there are already specific conditions $\mathcal{R}^X$ which, when satisfied by a total function $f$, guarantee that the general conditions $\mathcal{K}^{X\oplus f}$ remain sufficiently dense.
\end{remark}

\begin{remark}[Essential Requirements]
In the above construction, we used the fact that either you can define $V_0(x)$ by finding a $\hat d(x)(2n_\Xi-1)+1$-branching set of strings yielding convergent computations that aren't too big, or you can restrict yourself to strings that yield computations that diverge or are too big.
In the iteration construction, this combinatorial idea becomes a dichotomy; the requirements satisfying the convergent case will be called ``essential''.
\end{remark}

\subsection{Forcing Solutions of \DNR}\label{sec:dDNR-iteration}

We now turn to the general iterated forcing construction.  
Our setting in this subsection is that we have already constructed $X= V^\infty\oplus f^\infty_0 \oplus \dots \oplus f^\infty_m$. 
The ground forcing in the next section will show how to build $V^\infty$.  
Without loss of generality, we may assume that functions $f^\infty_i$ have already been built by iterating the construction in this section.

We now wish to construct a $\mathrm{DNR}^X$ function $f^\infty$ such that no $X\oplus f^\infty$-computable function ``diagonalizes against ${V^\infty}$'' (recall Definition \ref{defn.diagonalizes-against}).  
Because $V^\infty\leq_T X\oplus f^\infty$, it follows that that no $X\oplus f^\infty$-computable function solves $\dDNR^{X\oplus f^\infty}$. 

We will construct $f^\infty$ using Mathias forcing.
As in the one-step construction, we will write $V_0(x)$ to represent the Turing computation with oracle $V^\infty$ that outputs $i$ if $V^\infty(x) = \langle i,n\rangle$ for some $n\in \mathbb{N}$.

As in the one-step construction, it suffices to construct a $\mathrm{DNR}^{X}$ function $f^\infty$ such that for each $e$, either $\Phi_e^{X\oplus f^\infty}$ is not total or $\hat d$-bounded, or such that $\Phi_e^{X\oplus f^\infty}(x)\downarrow = i =V_0(x)$ for some $x$ and some $i<\hat d(x)$.

\begin{remark}
Notice that the set $V^\infty$, which we will build once in the ground forcing section, will remain the single witness that $X\oplus f^\infty$ does not compute a solution to $\dDNR^{X\oplus f^\infty}$ no matter how many times we iterate this construction to add new functions to $X$.
\end{remark}


\begin{definition}\label{defn.iteration-conditions}
We define $\mathbb{P}^X$ to be the set of triples $(f,n_\Xi,\Xi)$ such that $f$ is $\mathrm{DNR}^X$ and $\Xi$ is a family of extensions of $f$ which blocks at width $n_\Xi$.  We say $(f',n'_\Xi,\Xi')\preceq(f,n_\Xi,\Xi)$ if $f'\in\Xi$ and $\Xi'\subseteq \Xi$.

We say that a pair of $\Sigma_1^X$ families of strings
\[\mathcal{K}^{X,-}(x)=\{f\in \omega^{\subset\omega}\mid \exists y R^{X,-}(x,y,f)\}=\{f\in \omega^{\subset\omega}\mid \exists y \Phi_{e_0}^{X}(x,y,f)\}\]
\[\mathcal{K}^{X,+}(x,i)=\{f\in \omega^{\subset\omega}\mid \exists y R^{X,+}(x,y,i,f)\}=\{f\in \omega^{\subset\omega}\mid\exists y \Phi_{e_1}^{X}(x,y,i,f)\}\]
is a \emph{requirement} if $R^{X,-}$ and $R^{X,+}$ are $X$-computable relations such that 
\[\text{if }R^{X,-}(x,y,f)\text{, there is an }i\leq \hat d(x)\text{ such that }R^{X,+}(x,i,y,f).\]
In particular, requirements satisfy that $\K^{X,-}(x) \subseteq \bigcup_{i\leq \hat d(x)}\K^{X,+}(x,i)$.


We also define 
\[\mathcal{K}^{X,+}(x)=\begin{cases}
			\mathcal{K}^{X,+}(x,i)&\text{if }V_0(x)=i<\hat d(x)\\
			\mathcal{K}^{X,+}(x,\hat d(x))&\text{if }V_0(x)\geq \hat d(x)\text{ or }V_0(x)\uparrow
\end{cases}\]
\end{definition}
In order to force a requirement "negatively", we must arrange that all extensions avoid $\mathcal{K}^{X,-}(x)$ while in order to force a requirement "positively" we must place an extension in $\mathcal{K}^{X,+}(x)$.\footnote{The reader may wonder why we need to distinguish the positive and negative requirements, which was not needed in \cite{LST}.  In the constructions in \cite{LST}, we could look for enough extensions producing witnesses and then address all witnesses simultaneously.  Here, on the other hand, we will have to first pare down our collection of extensions to first, because it will not be possible to address all witnesses at once.}

Note that we also refer to $\mathcal{K}^X$, with no parameter, when we want to discuss the requirement itself, rather than the set of conditions it picks out.

We write $\mathcal{K}^X_{e_0,e_1}$ for the potential requirement where $R^{X,-}$ is $\Phi^X_{e_0}$, $R^{X,+}$ is $\Phi^X_{e_1}$.

\begin{example}The main example is the diagonalization requirement.

\begin{itemize}
	\item   $\mathcal{D}_{e,d}^{X,-}(x)$ consists of those $f$ such that $\Phi^{X\oplus f}_e(x)\downarrow$.  
	\item   For $i<\hat d(x)$, $\mathcal{D}_{e,d}^{X,+}(x,i)$ is the set of $f$ such that $\Phi^{X\oplus f}_e(x)\downarrow=i$.
	\item   If $i=\hat d(x)$, $\D_{e,d}^{X,+}(x,\hat d(x))$ is the set of $f$ such that $\Phi^{X\oplus f}_e(x)\downarrow\geq \hat d(x)$
	\item $\D_{e,d}^{X,+}(x) = \begin{cases}
																\D_{e,d}^{X,+}(x,i) & \text{if }V_0(x)\downarrow =i <\hat d(x)\\[.25cm]
																\D_{e,d}^{X,+}(x,\hat d(x)) & \text{if } V_0(x) \geq \hat d(x) \text{ or } V_0(x)\uparrow \\[.25cm]
														 \end{cases}$
\end{itemize}

	Therefore, if $V_0(x)\downarrow=i<\hat d(x)$, then $\D_{e,d}^{X,+}(x) $ is the set of $f$ such that $\Phi^{X\oplus f}_e(x)=i$.
	Otherwise, when $V_0(x)$ diverges or is at least $\hat d(x)$, then $\D_{e,d}^{X,+}(x) $ is the set of $f$ such that $\Phi^{X\oplus f}_e(x)\geq \hat d(x)$.
	
\end{example}

\begin{definition}
We say that a \emph{total} function $f^\infty:\omega\rightarrow\omega$ \emph{settles} $\mathcal{K}^X$ if there is some $x$ such that either there is a finite $f\subset f^\infty$ with $f\in\mathcal{K}^{X,+}(x)$, or for every finite $f\subset f^\infty$, $f\not\in\mathcal{K}^{X,-}(x)$.  (Note that $f$ need not be an initial segment of $f^\infty$).
\end{definition}

Our eventual goal is to build $f^\infty$ to settle all ``essential'' requirements $\K^X$ (as defined in Definition \ref{defn.dDNR-essential}).
First, we verify that settling each $\D^X$ is enough to avoid solving $\dDNR$.

\begin{lemma}
  Suppose $f^\infty$ settles $\mathcal{D}_{e,d}^X$ for every $e$.  Then $X\oplus f^\infty$ does not compute any $d\mhyphen\mathrm{DNR}^{V^\infty}$ function.
\end{lemma}
\begin{proof}
It suffices to show that for each $e$, $\Phi^{X\oplus f}_e$ does not diagonalize against $V^\infty$.  Since $f^\infty$ settles $\mathcal{D}_{e,d}^X$, we let $x$ be as in the definition of ``settles,'' and we consider the two cases.  

If there is some $f\subset f^\infty$ with $f\in\mathcal{K}^{X,+}(x)$, we must have $\Phi^{X\oplus f}_{e,t}(x)\downarrow$ at some stage $t$.
By the usual convention, this means that $\{0,\dots,t\}\subseteq dom(f)$, and therefore $f\upto t=f^\infty\upto t$.
Examining the definition of $\mathcal{D}^{X,+}$, we see there are two subcases.
If $V_0(x)$ is at least $\hat d(x)$ or diverges then $\Phi^{X\oplus f}_e(x)\geq \hat d(x)$.  
If $V_0(x)\downarrow<\hat d(x)$, then $\Phi^{X\oplus f}_e(x)=V_0(x)$.  
Either way, $\Phi^{X\oplus f^\infty}_e$ is either not $\hat d$ bounded or it does not diagonalize against $V_0(x)$.

On the other hand if every $f\subset f^\infty$ avoids $\mathcal{K}^{X,-}(x)$, it must be that $\Phi^{X\oplus f^\infty}_e(x)\uparrow$, so $\Phi^{X\oplus f^\infty}_e$ is not total.
\end{proof}

In the one-step construction, we saw that if you could not find a $\hat d(x)(2n_\Xi-1)+1$-branching set of one strings to force one outcome, then we can restrict to conditions that force the other outcome.
A similar idea works if there is \emph{any} limit on the amount of branching in the sets which yield the first outcome.
This suggests the following definition of which requirements ``essential.'' 

\begin{definition}\label{defn.dDNR-essential}
$\mathcal{K}^X$ is \emph{essential} below $(f,n_\Xi,\Xi)$ if for every $x$ and every $m$ there is an $m$-branching set of extensions of $f$, $U\subseteq \Xi\cap\mathcal{K}^{X,-}(x)$.

$\mathcal{K}^X$ is \emph{uniformly dense} if whenever $\mathcal{K}^X$ is essential below $(f,n_\Xi,\Xi)$, there is some $x$ and some $2n_\Xi$-branching set of extensions of $f$, $U\subseteq \mathcal{K}^{X,+}(x)$.
\end{definition}

\begin{remark}
Note that saying that $U\subseteq \K^{X,+}(x)$ is stronger than ``there is \emph{some} $i$ such that $U\subseteq \K^{X,+}(x,i)$,''
because the correct $i$ depends on $V_0(x)$. 
\end{remark}

Intuitively, ``$\K^X$ is uniformly dense below $q$'' means ``whenever you have many options to extend $q$ so that $\K^{X,-}$ holds, you will have one extension which allows you to force the stronger condition $\K^{X,+}$ to hold.''
In Theorem \ref{thm:settling_everything}, we will show that uniform density of a requirement $\K^X$ allows you to force either $\neg\K^{X,-}$ or $\K^{X,+}$.

\begin{lemma}\label{lemma-dDNR-iteration-extending-f}
Suppose $\mathcal{K}^X$ is uniformly dense.  Then for every $(f,n_\Xi,\Xi)$, there is an $(f',n'_\Xi,\Xi')\preceq(f,n_\Xi,\Xi)$ and an $x$ such that either $f'\in\mathcal{K}^{X,+}(x)$ or $\Xi'\cap\mathcal{K}^{X,-}(x)=\emptyset$.
\end{lemma}
\begin{proof}
If $\mathcal{K}^X$ is not essential below $(f,n_\Xi,\Xi)$ then for some $x$ and some $m$, whenever $U$ is an $m$-branching set of extensions of $f$, there is a $g\in U$ with $g\not\in\mathcal{K}^{X,-}(x)$.  Let $\Xi'\subseteq\Xi$ consist of those $g$ such that $g\in\Xi$ and $g\not\in\mathcal{K}^{X,-}(x)$.  
Then by Lemma \ref{tree_pigeonhole} and the definition of $U$, any $m+n_\Xi-1$-branching set contains an $n_\Xi$-branching subset of $U$.   
Because $\Xi$ is blocking at width $n_\Xi$, it follows that $\Xi'$ is blocking at width $m+n_\Xi-1$.
In particular, $(f,m+n_\Xi-1,\Xi')\preceq(f,n_\Xi,\Xi)$ satisfies the second case.

Otherwise, $\mathcal{K}^X$ is essential below $(f,n_\Xi,\Xi)$, so by uniform density, we may find some particular $x$ and some $2n_\Xi$-branching set of extensions $f$, $U\subseteq \mathcal{K}^{X,+}(x)$  
By Lemma \ref{thm:dnr_extension} we find a $2$-branching set of $g\in U\cap\mathcal{K}^{X,+}(x)$ and by Lemma \ref{lemma.branch_gives_DNR} we may find a $g$ so that $(g,n_\Xi,\Xi\upharpoonright g)$ is a condition below $(f,n_\Xi,\Xi)$
\end{proof}

\begin{theorem}\label{thm:settling_everything}
  Fix any countable collection of uniformly dense requirements $\mathcal{K}^X_k$.  Then there is a $\mathrm{DNR}^X$ function $f^\infty$ settling every $\mathcal{K}^X_k$.
\end{theorem}
\begin{proof}
  Enumerate the uniformly dense requirements $\mathcal{K}^X_0,\mathcal{K}^X_1,\ldots$.  Construct a sequence of conditions $(f^0,n^0_\Xi,\Xi^0)\preceq(f^1,n^1_\Xi,\Xi^1)\preceq\cdots$ by repeatedly applying the previous lemma so that for each $k$, there is an $x$ such that either $f^k\in\mathcal{K}^{X,+}_k(x)$ or $\Xi^k\cap\mathcal{K}^{X,-}_k(x)=\emptyset$.  Then letting $f^\infty=\bigcup_k f^k$, $f^\infty$ is $\mathrm{DNR}^X$ by the definition of $\mathbb{P}^X$ and settles each $\mathcal{K}^X_k$.
\end{proof}

In summary, we have seen that settling each $\D^X$ ensures we avoid solving $\dDNR^X$, and that all $\K^X$ being uniformly dense will allow us to settle each $\K^X$.
In the ground forcing construction, we will build $V^\infty$ so that all the requirements $\K^{V^\infty}$ are uniformly dense.

We next define requirements which, when satisfied, will ensure that requirements will remain uniformly dense with respect to $X\oplus f^\infty$.
This guarantees that the construction can be iterated to form a Turing ideal.

\begin{definition}
\label{defn.dDNR-bootstraping-req}
  Let $e_0,e_1$ be indices and let $q=(g,m_\Pi,\Pi)$ be a triple.  Then $\mathcal{R}_{e_0,e_1,q}^{X,-}(x)$ is the set of $f\in 2^{<\omega}$ such that either:
  \begin{itemize}
  \item $f$ forces that $\mathcal{K}^{X\oplus f}_{e_0,e_1}$ is not a requirement 
			$$\text{i.e. }(\exists x,y)[\Phi^X_{e_0}(x,y,f)\land(\forall i\leq \hat d(x))\neg \Phi^X_{e_1}(x,y,i,f)]\text{, or}$$
  \item $f$ forces that $q$ is not a condition, or
  \item There is a $\hat d(x)(2m_\Pi-1)+1$-branching set of extensions of $g$ contained in $\mathcal{K}_{e_0}^{X\oplus f,-}(x)$.
  \end{itemize}

 $\mathcal{R}_{e_0,e_1,q}^{X,+}(x,i)$ is the set of $f$ such that either:
  \begin{itemize}
  \item $f$ forces that $\mathcal{K}^{X\oplus f}_{e_0,e_1}$ is not a requirement, or
  \item $f$ forces that $q$ is not a condition, or
  \item There is a $2m_\Pi$-branching set of extensions of $g$ contained in $$\mathcal{K}_{e_1}^{X\oplus f,+}(x,i).$$
  \end{itemize}

Finally, we define \par
$\mathcal{R}_{e_0,e_1,q}^{X,+}(x)=\begin{cases}
			\mathcal{R}_{e_0,e_1,q}^{X,+}(x,i)&\text{if }V_0(x)=i<\hat d(x)\\
			\mathcal{R}_{e_0,e_1,q}^{X,+}(x,\hat d(x))&\text{if }V_0(x)\geq \hat d(x)\text{ or }V_0(x)\uparrow
			\end{cases}$
\end{definition}

Note that the definition of $\R^{X,-}$ references the amount of branching used in the one-step construction.
In the one-step construction, we worked inside a $\hat d(x)(2n_\Xi -1) +1$-branching set to obtain $2n_\Xi$ branching extensions which forced some computational outcome. 
Now, we will use the same combinatorial property to a more abstract version of the same idea: that if $f\in\R^{X,-}(x)$, then $f\in\R^{X,+}(x,i)$ for some $i$.
That is, it shows that the $\R^{X}$ are requirements. 

\begin{lemma}\label{lemma.dDNR-R-is-req}
  $\mathcal{R}_{e_0,e_1,q}^X$ is a requirement.
\end{lemma}
\begin{proof}
Clearly $\mathcal{R}_{e_0,e_1,q}^X$ consists of a pair of $\Sigma^0_1$ families of strings.  
We must show that if $\Phi^X_{e_0,q}(x,y,f)$ is true, then there is some $i<\hat d(x)$ such that $\Phi^X_{e_1,q}(x,y,i,f)$ is true.

Suppose $f\in\mathcal{R}_{e_0,e_1,q}^{X,-}(x)$ holds with existential witness $y$.  
We must show that there is some $i<\hat d(x)$ such that $f\in\mathcal{R}_{e_0,e_1,q}^{X,+}(x,i)$ with the same existential witness $y$.
If either of the first two cases hold, we trivially have $f\in\mathcal{R}_{e_0,e_1,q}^{X,+}(x,i)$ for all $i$ as well, and we are done.

Suppose $f$ is not in the first two cases, so there is a $\hat d(x)(2m_\Pi-1)+1$-branching set of extensions of $g$ contained in $\mathcal{K}_{e_0}^{X\oplus f,-}(x)$.  Since we are not in the first case, $f$ forces that $\mathcal{K}^{X\oplus f,-}_{e_0}(x)\subseteq\bigcup_{i< \hat d(x)}\mathcal{K}^{X\oplus f,+}_{e_1}(x,i)$ (with the same existential witness $y$).  Therefore, by Lemma \ref{lemma.dDNR-iteratedpigeon}, there is an $i$ and a $2m_\Pi$-branching set of extensions of $g$ contained in $\mathcal{K}^{X\oplus f,+}_{e_1}(x,i)$.
Consequently, $f\in\bigcup_{i<\hat d(x)}\mathcal{K}^{X\oplus f,+}(x,i)$.
\end{proof}

\begin{lemma}\label{lemma.dDNR-settling-R-preserves-density}
  If $f^\infty$ settles $\mathcal{R}_{e_0,e_1,q}^X$ for every $q$ and if $\mathcal{K}^{X\oplus f^\infty}_{e_0,e_1}$ is a requirement, then $\mathcal{K}^{X\oplus f^\infty}_{e_0,e_1}$ is uniformly dense.
\end{lemma}
\begin{proof}
  Let $q$ be a condition and suppose $\mathcal{K}^{X\oplus f^\infty}$ is essential below $q=(g,m_\Pi,\Pi)$.  
	We must show that there is a $2m_\Pi$-branching set of extensions $g$ contained in $\K_{e_1}^{X\oplus f,+}(x)$.
	
	Because no $f\subset f^\infty$ can force $q$ to be a non-condition or $\K^{X\oplus f}_{e_0,e_1}$ to be a non-requirement, 
		the definitions of $\R^{X,+}(x,i)$ and $\R^{X,+}(x)$ show that 
		we have the desired $2m_\Pi$-branching set of extensions $g$ which is contained in $\K_{e_1}^{X\oplus f,+}(x)$
		if and only if 
		$f\in \R^{X,+}(x)$. 
	It therefore suffices to find some $f\subset f^\infty$ such that $f\in \R^{X,+}(x)$.
	
	Fix some $x$ witnessing that $f^\infty$ settles $\R^X_{e_0,e_1,q}$.
	Then either there is a finite $f\subset f^\infty$ with $f\in \R^{X,+}(x)$ and we are done, or for every finite $f\subset f^\infty$, $f\notin \R^{X,-}(x)$.
	We will show the first case holds.
	
	By the definition of $\K^{X\oplus f^\infty}_{e_0,e_1}$ being essential below $q$, we know that for every $x$ there is a $\hat d(x)(2m_\Pi-1)+1$-branching set of extensions of $g$ in $\mathcal{K}^{X\oplus f^\infty,-}_{e_0}(x)$.  
	Because $\K$ has an existential definition, there is a finite $f\subset f^\infty$ such that there is a $\hat d(x)(2m_\Pi-1)+1$-branching set of extensions of $g$ in $\mathcal{K}^{X\oplus f,-}_{e_0}(x)$.  
	By the definition of $\R^{X,-}(x)$, we see that $f\in \R^{X,-}(x)$.
\end{proof}

\begin{theorem}\label{thm:dnr_iterate}
  Suppose $V^\infty\leq_T X$ are such that every requirement $\mathcal{K}^X_{e_0,e_1}$ is uniformly dense.  Then there is a $\mathrm{DNR}^X$ function $f^\infty$ such that no $\Phi^{X\oplus f^\infty}_e$ is $d\mhyphen\mathrm{DNR}^{V^\infty}$ and every requirement $\mathcal{K}^{X\oplus f^\infty}_{e_0,e_1}$ is uniformly dense.
\end{theorem}

\subsection{Ground Forcing}\label{sec:dDNR-Ground}

Fix a computable function $d$.  We now construct a set which can serve as the initial $X=V^\infty$ in the application of the forcing above.  The main goal is to ensure that all requirements $\K$ are uniformly dense.  More intuitively, our goal is to construct $V^\infty$ in such a way as to obstruct the construction of a $d\mhyphen\mathrm{DNR}^{V^\infty}$ set.  

As before, we do this by making $V^\infty$ a partial function mapping values $x$ to pairs $(i,n)$ with $i<\hat d(x)$.  
Let $V_0^\infty$ be the partial $V^\infty$ computable function given by $V_0^\infty(x)\downarrow=i$ iff $(\exists n)[V^\infty(x)=(i,n)]$.
We wish to force that any $d\mhyphen\mathrm{DNR}^{V^\infty}$ function $F$ would have to have the property that $F(x)\neq V_0^\infty(x)$ whenever $V_0^\infty(x)$ is defined. 
Unfortunately, we will want to reference computations from certain requirements $\K^{V^\infty}_{e_0,e_1}$ when selecting the value of $V_0^\infty$.
To avoid circularity, we will define $V^\infty(x)=(i,n)$, where $n$ is larger than the use of any computation performed so far, to ensure that our definition of $V_0^\infty$ does not change the relevant computations involving $\K^{V^\infty}_{e_0,e_1}$.

We build $V^\infty$ by forcing with approximations, which are tuples $(m,V,R)$ such that:
\begin{itemize}
\item $V$ is a partial function on $[0,m]$ such that when $x\in\dom(V)$, $V(x)=(i,n)$ for some $i<\hat d(x)$ and $n\leq m$.
\item $R\subseteq[0,m]$,
\item $R\cap\dom(V)=\emptyset$.
\end{itemize}
Given $V$, we write $V_0$ and $V_1$ for the projection functions, so $\dom(V_0)=\dom(V_1)=\dom(V)$ and $V(x)=(V_0(x),V_1(x))$.  We say $(m',V',R')\preceq(m,V,R)$ if $m'\geq m$, $V\subseteq V'$, $x\in\dom(V')\setminus\dom(V)$ implies $V_1(x)>m$, and $R\subseteq R'$.  We write $|V|$ for the smallest $m$ such that $\dom(V)\subseteq[0,m]$ and for each $x\in\dom(V)$, $V_1(x)\leq m$.  We write $V'\preceq(m,V,R)$ if $(|V'|,V,R)\preceq(m,V,R)$.

\begin{lemma}\label{lemma.dDNR-ground-iteration}
For any $(m,V,R)$, any $e_0,e_1$, and any $p=(f,n_\Xi,\Xi)$, there is an $(m',V',R')\preceq(m,V,R)$ such that either:
\begin{itemize}
\item $(m',V',R')$ forces that $\mathcal{K}^{V^\infty}_{e_0,e_1}$ is not a requirement,
				$$\text{i.e. }(\exists x,y)[\Phi^{V^\infty}_{e_0}(x,y,f)\land(\forall i<\hat d(x))\neg \Phi^{V^\infty}_{e_1}(x,y,i,f)]\text{, or}$$
\item $(m',V',R')$ forces that $p$ is not a condition, or
\item $(m',V',R')$ forces that $\mathcal{K}^{V^\infty}_{e_0,e_1}$ is uniformly dense below $p$.
\end{itemize}
\end{lemma}
\begin{proof}
If any $(m',V',R')\preceq(m,V,R)$ forces either of the first two cases, we are finished, so suppose not.  We ask:
\begin{quote}
Is there any $V'\preceq(m+1,V,R\cup\{m+1\})$ such that there is a $\hat d(x)(2n_\Xi-1)+1$-branching set of extensions $U$ of $f$ in $\mathcal{K}^{V',-}_{e_0,e_1}(m+1)$?
\end{quote}

Suppose not.  Then $(m+1,V,R\cup\{m+1\})$ forces that $\mathcal{K}^{V^\infty}$ is not essential below $p$.  

Otherwise, we may find an appropriate $V'\preceq(m+1,V,R\cup\{m+1\})$ and a set of extensions $U\subset \K^{V',-}_{e_1,e_2}(m+1)$.  
We must show that there is some $x$ and some $2n_\Xi$-branching set $U'$ of extensions of $f$ which is contained in $\K_{e_0,e_1}^{X,+}(x)$. 
Note that (by the definition of condition) for each $g\in U$, there is an $i\leq \hat d(x)$ with $g\in\mathcal{K}_{e_0,e_1}^{V',+}(m+1,i)$.  
By Lemma \ref{lemma.dDNR-iteratedpigeon}, for some $i$ there is a $2n_\Xi$-branching $U_i\subseteq U\cap\mathcal{K}^{V',+}(m+1,i)$.

We now extend to the condition $(|V'|+1,V'\cup\{(m+1,(i,|V'|+1))\},R)$. 
Notice this forces that $V_0(m+1)=i$. 
If $i<\hat d(m+1)$, then the definition of $\K^{X,+}$ gives $\K_{e_0,e_1}^{X,+}(m+1)= \mathcal{K}^{V',+}(m+1,i)$.  
If $i=\hat d(m+1)$, then trivially $V_0(m+1)\geq \hat d(m+1)$, and so $\K_{e_0,e_1}^{X,+}(m+1)= \mathcal{K}^{V',+}(m+1,\hat d(m+1))$.  
Either way, setting $U'=U_i\subset \mathcal{K}^{V',+}(m+1,i)=\K_{e_0,e_1}^{X,+}(m+1)$ we see that the uniform density condition holds with respect to the current oracle $V'$.

We wish to show that all extensions of the condition  $(|V'|+1,V'\cup\{(m+1,(i,|V'|+1))\},R)$ continue to satisfy the uniform density condition for $\mathcal{K}^{V^\infty}_{e_0,e_1}$ below the condition $p$.
To do this, we must show that extending to future conditions will not change any of the computations used to witness the uniform density condition.  
But recall that all computations used to find the $2n_\Xi$-branching $U'\subseteq U\cap\mathcal{K}^{V',+}(m+1,i)$ have had use at most $|V'|$.  
By choosing $m'=|V'|+1$ and of $V_1^\infty=|V'|+1$, we respect restraints from previous stages by defining $V(m+1)=(i,|V'|+1)$ (i.e.\ setting $V_0^\infty(m+1)=i$ at stage $|V'|+1$), and we move the restraint to preserve $V^\infty$ below $|V'|+1$.

In other words, $U'$ continues to witness the uniform density condition in every extension of $V''$ of $(|V'|+1,V'\cup\{(m+1,(i,|V'|+1))\},R)$. 
\end{proof}

\begin{theorem}
  For any computable $d$, there is a Turing ideal $\mathcal{I}$ satisfying $\DNR$ but not $d\mhyphen\DNR$.
\end{theorem}
\begin{proof}
  Applying the previous lemma repeatedly, we may construct a $V^\infty$ so that every requirement $\mathcal{K}^{V^\infty}$ is uniformly dense.  By Theorem \ref{thm:dnr_iterate} we may find a $\mathrm{DNR}^{V^\infty}$ function $f^\infty$ so that no $\Phi^{V^\infty\oplus f^\infty}_e$ is $d\mhyphen\mathrm{DNR}^{V^\infty}$ and every requirement $\mathcal{K}^{V^\infty\oplus f^\infty}$ is uniformly dense.   Iterating this countably many times and closing under computability, we obtain a Turing ideal $\mathcal{I}$ containing $V^\infty$, containing no $d\mhyphen\mathrm{DNR}^{V^\infty}$ function, and so that for every $W\in\mathcal{I}$, $\mathcal{I}$ contains a $\mathrm{DNR}^W$ function.
\end{proof}

\begin{cor}\ 
  \begin{itemize}
  \item For any computable $d$, $\DNR$ does not imply $d\mhyphen\DNR$,
  \item $\DNR$ does not imply $\WWKL$.
  \end{itemize}
\end{cor}

\section{\DNR{} does not imply \RKL}\label{sec:rkl}

\subsection{The one step case}
Suppose we wish to construct an infinite $\{0,1\}$-tree $T^\infty$ and a function $f^\infty$ which is $\DNR^{T^\infty}$ so that no $W_e^{f^\infty}$ is a solution to the $\RKL$ instance $T^\infty$.  That is, either $W_e^{f^\infty}$ is finite or there is a level $n$ such that whenever $\sigma\in T^\infty$ with $|\sigma|=n$, there are $x,z\in W_e^{f^\infty}$ with $\sigma(x)=1-\sigma(z)$.

\begin{definition}
  A \emph{finite $\{0,1\}$-tree} is a set $T\subseteq\{0,1\}^{< n}$ such that if $\sigma\in T$ and $\tau\sqsubseteq\sigma$ then $\tau\in T$.  We say $n=|T|$ is the \emph{height} of $T$ if $n$ is the length of the longest sequence in $\sigma$.  
  
If $\sigma$ is a sequence, the \emph{reflection} of $\sigma$ above $x$, written $\sigma_{/x}$, is the unique sequence such that
\[\sigma_{/x}(y)=\left\{\begin{array}{ll}
\sigma(y)&\text{if }y\leq x\text{ and }y<|\sigma|\\
1-\sigma(y)&\text{if }x<y<|\sigma|\\
\end{array}\right.\]

If $T$ is a finite $\{0,1\}$-tree of height $n$ and $S\subseteq[0,n]$, we say $T$ is \emph{symmetric} over $S$ if whenever $\sigma\in T$ and $x\in S$, $\sigma_{/x}\in T$.
\end{definition}
That is, $\sigma_{/x}$ has the same length as $\sigma$, but reverses the value of every element larger than $x$.

We will force with tuples
\[(T,S,f,n_\Xi,\Xi)\]
where
\begin{itemize}
\item $T$ is a finite $\{0,1\}$-tree,
\item $S$ is a finite set such that $T$ is symmetric across $S$,
\item $f\in\omega^{\subset\omega}$ is $\mathrm{DNR}^T$,
\item $\Xi$ is a function so that when $T'$ extends $T$ and is symmetric across $S$, then $\Xi(T')$ is a set of extensions of $f$ intersecting every $n_\Xi$-branching set of extensions of $f$, and
\item If $T''$ extends $T'$ and $T'$ extends $T$ then $\Xi(T'')\subseteq \Xi(T')$.
\end{itemize}
We say $(T',S',f',n'_\Xi,\Xi')\preceq(T,S,f,n_\Xi,\Xi)$ if 
\begin{itemize}
\item $T'$ extends $T$,
\item $S\subseteq S'$,
\item $f',n'_\Xi,\Xi'$ extend $f,n_\Xi,\Xi$ as in the previous section.
\end{itemize}
We write $T'\preceq(T,S)$ if $T'$ extends $T$ and is symmetric across $S$.

Let $(T,S,f,n_\Xi,\Xi)$ be given and suppose we wish to force that $W_e^{T^\infty\oplus f^\infty}$ is not a solution to $T^\infty$.  We proceed in two steps.  First we want to force $W_e^{T^\infty\oplus f^\infty}$ to either enumerate an element greater than any element of $S$ or be finite.  We fix a value $x$ larger than $|T|$ and larger than any value in $S$.  We ask:
\begin{quote}
  Is there some $T'\preceq(T,S\cup\{x\})$ and some $6n_\Xi-2$-branching set of extensions $U$ of $f$ such that for each $g\in U$, $W_e^{T'\oplus g}\cap(x,|T'|]\neq\emptyset$?
\end{quote}
Suppose not.  Then for each $T'\preceq(T,S\cup\{x\})$, let $\Xi'(T')$ be the set of $g\in \Xi(T)$ such that $W_e^{T'\oplus g}\cap(x,|T'|]=\emptyset$.  
By choosing all future $g\in\Xi'(T')$, it follows that $W_e^{T^\infty\oplus f^\infty}\subseteq[0,x]$, and is therefore finite.
We claim that
\[(T,S,f,7n_\Xi-3,\Xi')\preceq(T,S,f,n_\Xi,\Xi).\]
To see that $\Xi'(T')$ is blocking at width $7n_\Xi-3$, recall by Lemma \ref{tree_pigeonhole} that any $6n_\Xi-2+n_\Xi-1$-branching set $U$ of strings either has a $6n_\Xi-2$-branching subset of $\{g: W_e^{T'\oplus g}\cap(x,|T'|]\neq\emptyset\}$ or a $n_\Xi$-branching subset of $\{g : W_e^{T'\oplus g}\cap(x,|T'|]=\emptyset\}$.

So suppose we find $T'$ and $U$.  We pick a $z$ larger than $x$ and larger than $|T'|$ and ask:
\begin{quote}
  Is there some $T''\preceq(T',S\cup\{x,z\})$ and some $4n_\Xi-1$-branching set of extensions $U'\subseteq U$ of $f$ such that for each $g\in U'$, $W_e^{T''\oplus g}\cap (x,z]\neq\emptyset$ and $W_e^{T''\oplus g}\cap(z,|T''|]\neq\emptyset$?
\end{quote}
Suppose not.  
By Lemma \ref{tree_pigeonhole}, because $U$ is $(4n_\Xi-1)+(2n_\Xi)-1$ branching, if it does not (for some $T''$) contain a $4n_\Xi-1$-branching subset of such extensions, then it contains (for that $T''$) a $2n_\Xi$-branching subset of $g$ s.t. either $W_e^{T''\oplus g}\cap (x,z]=\emptyset$ or $W_e^{T''\oplus g}\cap(z,|T''|]=\emptyset$.  
By our choice of $T',U,$ and $z$, $W_e^{T''\oplus g}\cap(x,|T'|]\subseteq W_e^{T''\oplus g}\cap(x,z]$ is non-empty, so we are in fact guaranteed for each $T''$ a $2n_\Xi$-branching set $U^*_{T''}$ such that $W_e^{T''\oplus g}\cap(z,|T''|]=\emptyset$ for each $g\in U^*_{T''}$.
As in Subsection \ref{sec:onestep}, without loss of generality we may assume that if $g\in U\setminus U^*_{T'}	$ and $T''\preceq(T',S\cup\{x,z\})$ then $g\in U\setminus U^*_{T''}$.  (When this is not the case, we extend $T'$ finitely many times to obtain $T'$ for which this is true.)  Then by Lemma \ref{thm:dnr_extension} there is a $g\in U^*_{T'}$ so that $(T',S\cup\{x,z\},g,n_\Xi,\Xi\upharpoonright g)$ is a condition.  Finally, for every $T''\preceq(T',S\cup\{x,z\})$ we set $\Xi'(T'')$ to be those $h\in \Xi(T'')$ such that $W_e^{T''\oplus h}\cap(z,|T''|]=\emptyset$.  
Since $g\in U^*_{T'}=U^*_{T''}$, and since by assumption no $4n_\Xi-1$-branching set $V$ avoids $U^*_{T''}$, each $(4n_\Xi-1)+(n_\Xi)-1$-branching set contains a $n_\Xi$-branching subset of $V\cap U^*_{T''}$ and so $\Xi\upto g$ contains some element of $V$.  In other words, $\Xi'$ blocks at width $(4n_\Xi-1)+(n_\Xi)-1$, so $(T',S\cup\{x,z\},g,5n_\Xi-2,\Xi')$ forces that $W_e^{T^\infty\oplus f^\infty}\subseteq[0,z]$.

So suppose we find $T''$ and $U'$.  
This is the main case, and where symmetry becomes essential.
Consider some $\sigma\in T''$ with $|\sigma|=|T''|$.  
Since we have symmetry in $T''$ over both $x$ and $z$, there are four obvious reflections of $\sigma$ in $T''$---$\sigma$ itself, $\sigma_{/x}$, $\sigma_{/z}$, and $\sigma_{/x/z}=\sigma_{/z/x}$.  
We wish to find a symmetric pair of these four strings witnessing that $W_e^{T''\oplus g}$ is not homogeneous for either of our chosen strings. 

By our choice of $T''$ and $U'$, for each $g\in U'$ we have values $a\in W_e^{T''\oplus g}\cap(x,z]$ and $c\in W_e^{T''\oplus g}\cap(z,|T''|]$.
Because $x<a<z<c$, we either have that $\sigma(a)\neq \sigma(c)$ and $\sigma_{/x}(a)\neq \sigma_{/x}(c)$ or we have that $\sigma(a)=\sigma(c)$ and hence both $\sigma_{/z}(a)\neq \sigma_{/z}(c)$ and $\sigma_{/x/z}(a)\neq\sigma_{/x/z}(c)$.
In either case, we preserve exactly the correct two of these reflections, allowing us to preserve the symmetry over $x$ while satisfying the requirement that $W_e^{T''\oplus g}$ fail to be homogeneous for the chosen reflections.

For the full construction, we will need to express this argument in terms of requirements $\K$.
For each $a\in D_0= (x,z]$ we will define a set $Y_a$ which consists of possible values for $c\in D_1=(z,|T|]$.
Then $\mathcal{K}(\{Y_a\})$ will consist of those $g$ such that for some $c\in W_e^{T''\oplus g}\cap(z,|T''|]$, $c\in Y_a$.
The full requirement (quantifying over all the possible $a$ in $D_0$) will be $\K(\{Y_a\}_{a\in D_0})$.
In other words, the full requirement $\mathcal{K}(\{Y_a\}_{a\in D_0})$ consists of the $g$ such that for some $c\in W_e^{T''\oplus g}\cap(z,|T''|]$ and some $a\in W_e^{T''\oplus g}\cap(x,z]$, $c\in Y_a$.

The next lemma proves that as long as $\K$ satisfies a reasonable property, then we can always find a $2_\Xi$-branching refinement $U^*$ of $U$ so that for any $g^*\in U^*$, then $W_e^{T^*\oplus g^*}$ is not homogeneous for any string of length $\sigma$ in $T^*$.
In the more abstract language of $\K$, this means that $U^*$ must be a subset of $\mathcal{K}(\{c\in D_1\mid \sigma(c)\neq \sigma(a)\}_{a\in D_0})$.

\begin{lemma}\label{thm:sym_breaking}
  Suppose $T$ is symmetric over $S\cup\{x,z\}$ with $s<x<z$ for all $s\in S$.  
	Let $D_0=(x,z]$ and $D_1=(z,|T|]$.  
	Let $U$ be a $4n_\Xi-1$-branching set of extensions.
	Furthermore assume, for each $g\in U$, that if there are two sequences of sets $\{Y_{a,i}\}_{a\in D_0}$ s.t. $Y_{a,0}\cup Y_{a,1}=D_1$ for each $a$, then $g\in\K(\{Y_{a,0}\})$ or $g\in\K(\{Y_{a,1}\})$.  

Then there is a $T^*\preceq(T,S\cup\{x\})$ and a $2n_\Xi$-branching $U^*\subseteq U$ such that for every $\sigma\in T^*$ with $|\sigma|=|T^*|$, $U^*\subseteq\mathcal{K}(\{c\in D_1\mid \sigma(c)\neq \sigma(a)\}_{a\in D_0})$.
\end{lemma}
\begin{proof}
Fix some $\sigma\in T$ with $|\sigma|=|T|$, and let $\tau_i=\sigma\upto D_i$ for $i\in\{0,1\}$.  Without loss of generality, we may assume that for \emph{every} $\rho\in T$ with $|\rho|=|T|$, and for each $i\in\{0,1\}$, $\rho\upharpoonright D_i\in\{\tau_i,1-\tau_i\}$.  (Otherwise, we can extend $|T|$ by one, extending exactly those $\sigma$ satisfying this property.)  By our choice of $\sigma$, we know that there are at least four strings in $T$: $\sigma$, $\sigma_{/x}$, $\sigma_{/z}$, and $\sigma_{/x/z}$.

For each $a\in D_0$, let $Y_{a,0}=\{c\in D_1\mid\tau_1(c)=\tau_0(a)\}$ and $Y_{a,1}=D_1\setminus Y_{a,0}=\{c\in D_1\mid \tau_1(c)\neq \tau_0(a)\}$.  
Clearly $Y_{a,0}\cup Y_{a,1}=D_1$ for each $a$, so by our assumption about $\K$, for each $g\in U$ we have either $g\in \mathcal{K}(\{Y_{a,0}\})$ or $g\in\mathcal{K}(\{Y_{a,1}\})$.
By Lemma \ref{lemma.dDNR-iteratedpigeon}, it follows that there is a $j\in\{0,1\}$ and a $2n_\Xi$-branching $U^*\subseteq U$ with $U^*\subseteq\mathcal{K}(\{Y_{a,j}\})$.  
It remains to extend $T$ to $T^*$ so that if $\sigma^*\in T^*$ then $Y_{a,j}=\{c\in D_1\mid \sigma^*(c)\neq \sigma^*(a)\}$ for each $a$.

If $j=0$, define $T^*\preceq(T,S\cup\{x\})$ so that $|T^*|=|T|+1$ and $\sigma^\frown\langle b\rangle\in T^*$ if $|\sigma|=|T|$, $\sigma\in T$, $b\in\{0,1\}$, and either 
\begin{itemize}
\item $\sigma\upharpoonright D_0=\tau_0$ and $\sigma\upharpoonright D_1=\tau_1$, or
\item $\sigma\upharpoonright D_0=1-\tau_0$ and $\sigma\upharpoonright D_1=1-\tau_1$.
\end{itemize}
In this case, when $\sigma^*\in T^*$, $a\in D_0$, and $c\in D_1$, $\sigma^*(a)=\tau_0(a)$ iff $\sigma^*(c)=\tau_1(c)$, so $\sigma^*(a)\neq\sigma^*(c)$ iff $c\in Y_{a,0}$, so $Y_{a,0}=\{c\in D_1\mid\sigma^*(c)\neq\sigma^*(a)\}$.

If $j=1$ we similarly define $T^*$ so that $\sigma^\frown\langle b\rangle\in T^*$ if $|\sigma|=|T|$, $\sigma\in T$, $b\in\{0,1\}$, and either 
\begin{itemize}
\item $\sigma^*\upharpoonright D_0=\tau_0$ and $\sigma\upharpoonright D_1=1-\tau_1$, or
\item $\sigma^*\upharpoonright D_0=1-\tau_0$ and $\sigma\upharpoonright D_1=\tau_1$.
\end{itemize}
Then if $\sigma^*\in T^*$, $a\in D_0$, and $c\in D_1$, $\sigma^*(a)=\tau_0(a)$ iff $\sigma^*(c)=1-\tau_1(c)$, so $\sigma^*(a)\neq\sigma^*(c)$ iff $c\in Y_{a,1}$, so $Y_{a,1}=\{c\in D_1\mid\sigma^*(c)\neq\sigma^*(a)\}$.
\end{proof}

In the one step case, recall that $\mathcal{K}(\{Y_a\}_{a\in D_0})$ is the set of $g$ such that there is some $c\in W_e^{T''\oplus g}\cap(z,|T''|)$ and some $a\in W_e^{T''\oplus g}\cap(x,z]$ with $c\in Y_a$.
Therefore, $\K$ satisfies the conditions of the lemma because each $c\in Y_{a,0}\cup Y_{a,1}$.

Applying the lemma to $T''$ and $U'$, we obtain $T^*$ and $U^*$.  $U^*$ is $2n_\Xi$-branching, so we may find a $g\in U^*$ so that $(T^*,S,g,n_\Xi,\Xi\upharpoonright g)$ is a condition.  Then for every $\sigma\in T^*$ with $|\sigma|=|T^*|$, there are $a,c\in W_e^{T^*\oplus g}$ so that $\sigma(c)\neq \sigma(a)$, and therefore $W_e^{T^\infty\oplus f^\infty}$ is not a solution to $T^\infty$.

\subsection{Forcing Solutions of \DNR}

We turn to the iterated forcing case for \RKL.  Our setting is that we have constructed a set $T^\infty\leq_T X$ and wish to construct a $\mathrm{DNR}^X$ function which does not compute a solution to $T^\infty$.  

The set of conditions $\mathbb{P}^X$ is the same as in Definition \ref{defn.iteration-conditions} of the previous section: we force with triples $(f,n_\Xi,\Xi)$ where $f$ is $\mathrm{DNR}^X$ and $\Xi$ is a family of extensions of $f$ which blocks at width $n_\Xi$.

Our requirements, however, are different.  
As illustrated in the one-step case, we need to force in two phases in order to accommodate the diagonalization requirement, which requires that we first enumerate an element into $W_e^{T^\infty\oplus f^\infty}$ and then enumerate a second element.

The requirement $\mathcal{K}^{X,-,0}(x)$ corresponds to enumerating the element $x$ into $W_e^{T^\infty\oplus f^\infty}$.
The requirement $\mathcal{K}^{X,-,1}(x,z)$ corresponds to enumerating the element $z$ into $W_e^{T^\infty\oplus f^\infty}$ after already enumerating $x$ into $W_e^{T^\infty\oplus f^\infty}$.

We first give the formal definitions, and then define the more intuitive basic requirement in Example \ref{example.RKL-condition-idea}.

\begin{definition}\label{defn.RKL-Requirement}
We say that three families of sets 
\[\mathcal{K}^{X,-,0}(x)=\{f\mid\exists y \exists f'\subseteq f\ R^{X,-,0}(x,y,f')\},\]
\[\mathcal{K}^{X,-,1}(x,z)=\{f\mid\exists y\exists f'\subseteq f\  R^{X,-,1}(x,z,y,f')\}\text{, and}\]
\[\mathcal{K}^{X,+}(x,z,\{Y_a\}_{a\in(x,z]})=\{f\mid\exists y\exists f'\subseteq f\exists Y'_a\subseteq Y_a\ R^{X,+}(x,z,y,\{Y'_a\},f')\}\]
and a partial computable function $\z^X_{\mathcal{K}}$ define an \emph{RKL-requirement} if the following conditions hold:
\begin{enumerate}
\item $R^{X,-,0},R^{X,-,1},R^{X,+}$ are $X$-computable relations which always terminate in at most $y$ steps,
\item If $f'\subseteq f$ and $\z^X_{\mathcal{K}}(f')\downarrow$ then $\z^X_{\mathcal{K}}(f)=\z^X_{\mathcal{K}}(f')$,
\item If $R^{X,-,0}(x,y,f)$ then $\z^X_{\mathcal{K}}(f)$ converges in $\leq y$ steps,
\item Suppose that $R^{X,-,0}(x,y,f)$, that $z\geq \z^X_{\mathcal{K}}(f)$, that $R^{X,-,1}(x,z,y',f)$, and that for each $a\in(x,z]$, $Y_{a,0}\cup Y_{a,1}$ is a partition of $(z,y')$.  
		\par Then there is a $j\in\{0,1\}$ so that $R^{X,+}(x,z,y',\{Y_{a,j}\},f)$.
\end{enumerate}
\end{definition}
The first three properties guarantee uniformity.  
The last property ensures we will be able to apply Lemma \ref{thm:sym_breaking} to the condition $\K^{X,+}(x,z,\{Y_a\}_{a\in(x,z]})$.

In the simplest case $\z^X_{\K}(f)$ will be a uniformly chosen element of $W^{X\oplus f}_e$.
The assertion ``$z\geq\z^X_{\K}(f)>x$'' is an abstraction of the assertion ``$W^{X\oplus f}_e\cap (x,z]\neq\emptyset$,'' and is more appropriate for use in a general condition.

\begin{definition}
We set $\mathcal{K}^{X,-}(x,z)$ to be those $f\in\mathcal{K}^{X,-,0}(x)$ such that either $z<\z^X_{\mathcal{K}}(f)$ or $f\in\mathcal{K}^{X,-,1}(x,z)$.
\end{definition}

Note that $f\notin\mathcal{K}^{X,-}(x,z)$ implies that either $f\notin \mathcal{K}^{X,-,0}(x)$, or $z\geq \z^X_{\mathcal{K}}(f)$ and $f\notin\mathcal{K}^{X,-,1}(x,z)$.
This is used in the second half of Lemma \ref{lemma.settles-W-does-not-solve-T}, where we examine both cases to show that $f^\infty$ strongly avoiding a certain $\K^{X,-}(x,z)$ (defined in Example \ref{example.RKL-condition-idea}) guarantees that $W_e^{X^\infty\oplus f^\infty}\subset[0,z]$.

Of course, we are most interested in strings in $\K^{X,-}(x,z)$ that are also in $\K^{X,-,1}(x,z)$.  
We formalize this in Definition \ref{defn.RKL-essential-and-dense}, where a requirement is only ``essential'' if a reasonable number of strings can be bootstrapped from $\K^{X,-,0}(x)$ to $\K^{X,-,1}(x,z)$.

\begin{definition}
When the $Y_a$ are (possibly) infinite sets, we write 
	$$\mathcal{K}^{X,+}(x,z,\{Y_a\}_{a\in(x,z]})=\bigcup \mathcal{K}^{X,+}(x,z,\{Y'_a\}_{a\in(x,z]})$$
where the union ranges of choices of finite $Y'_a\subseteq Y_a$.
Given $z$, we define $$Y_{a}(T^\infty)=\{ c>z : \text{ whenever }\Lambda\text{ is an infinite path in }T^\infty\text{, }\Lambda(c)\neq\Lambda(a)\}.$$  

When (and only when) $z\geq \z^X_{\K}(f)$ we set 
$$\mathcal{K}^{X,+}(x,z) = \{f \in \K^{X,-,0}(x) : f\in \K^{X,+}(x,z,\{Y_{a}(T^\infty)\}_{a\in(x,z]})\}$$
\end{definition}

We see in the ground forcing Lemma \ref{lemma.groundforcing} that we can move from not avoiding $\K^{X,-}(x,z)$ to being inside $\K^{X,+}(x,z)$.  
More formally, Lemma \ref{lemma.groundforcing} will show that we can force all requirements to be ``uniformly dense.''
This will be possible because in the ground forcing, we can \emph{choose} $z$ as large as needed.  

\begin{remark}
To understand the use of $\{Y_a\}$ in the definition of $\K^{X,+}(x,z)$, recall that we need the definition of a requirement to be $\Sigma^0_1$, so that we can easily force whether something is a requirement during the ground forcing.  
The definition of $\K^{X,+}(x,z)$ highlights that we will, during the ground forcing, compute $\{Y_{a}(T^\infty)\}_{a\in(x,z]}$ directly, and then provide it via the parameter $\{Y_a\}$ to obtain the correct $\Sigma^0_1$ class defined by $\K^{X,+}(x,z,\{Y_a\})$.
\end{remark}

Given $e_0,e_1,e_2,e_3$, we define $\mathcal{K}^X_{e_0,e_1,e_2,e_3}$ to be the potential requirement with $R^{X,-,0}_{e_0,e_1,e_2,e_3}=\Phi^X_{e_0}$, $R^{X,-,1}=\Phi^{X}_{e_1}$, $R^{X,+}_{e_0,e_1,e_2,e_3}=\Phi^X_{e_2}$, and $\z^X_{e_0,e_1,e_2,e_3}=\Phi^X_{e_3}$.

\begin{example}\label{example.RKL-condition-idea}
The requirement $\mathcal{W}^X_e$ is given by:
\begin{itemize}
\item  $\mathcal{W}^{X,-,0}_{e}(x)=\{f : W^{X\oplus f}_e$ contains an element $>x\}$,
\item $\z^X_{\mathcal{W}_e}(f)$ is the first value $>x$ enumerated into $W^{X\oplus f}_e$,
\item $\mathcal{W}^{X,-,1}_e(x,z)=\{f: \z^X_{\W_e}(f)$ witnesses that $W^{X\oplus f}_e\cap(x,z]\neq\emptyset$ and $W^{X\oplus f}_e$ contains an element $>z\}$, and 
\item $\mathcal{W}^{X,+}(x,z,\{Y_a\})$ consists of those $f$ such that there is some $a\in W^{X\oplus f}_e\cap(x,z]$ such that there is a $c\in W^{X\oplus f}_e\cap Y_a$ with $c>z$.
\end{itemize}

Then $\W^{X,-}(x,z)$ is the set of $f$ such that $W_e^{X\oplus f}$ contains an element greater than $x$ where \emph{either} $z$ is smaller than the first value $>x$ enumerated into $W_e^{X\oplus f}$ \emph{or} $W_e^{X\oplus f}\cap (x,z]\neq \emptyset$ and $W_e^{X\oplus f}$ contains an element $>z$.

In particular, as long as there is some element $>x$ in $W_e^{X\oplus f}$, then $f\notin \W^{X,-}(x,z]$ means that $\z^X_{\W_e}(f)$ is a uniform witnesses that $W_e^{X\oplus f}\cap (x,z]\neq 0$ \emph{and} that $W_e^{X\oplus f}\subseteq [0,z]$. 
	This is the main use of $\z^X_{\W_e}$: it provides an abstract way to say $W_e^{X\oplus f}\cap (x,z]\neq\emptyset$.
	
Finally, $\W^{X,+}(x,z)$ is the set of those $f$ such that there is some $a\in W^{X\oplus f}_e\cap(x,z]$ and some $c\in W^{X\oplus f}_e\cap(z,\infty)$ where $(\forall \Lambda\in[T^\infty])[\Lambda(c)\neq\Lambda(a)\}]$.
In other words, $W_e^{X\oplus f^\infty}$ is not a solution to $T^\infty$ for any $f^\infty\supseteq f$.
Note that in both $\W^{X,+}(x,z,\{Y_a\})$ and $\W^{X,+}(x,z)$, we also require that $W^{X\oplus f}_e\cap(x,z]\neq\emptyset$ be witnessed by observing that $\z^X_{\W_e}(f)\in W^{X\oplus f}_e\cap(x,z]$.
\end{example}

\begin{definition}
  We say $f^\infty:\omega\rightarrow\omega$ \emph{settles} $\mathcal{K}^X$ if there is some pair $(x,z)$ such that either there is a finite $f\subset f^\infty$ with $f\in\mathcal{K}^{X,+}(x,z)$ or for every finite $f\subset f^\infty$, $f\not\in\mathcal{K}^{X,-}(x,z)$.
\end{definition}

The next lemma unpacks and applies the definitions of $\W^{X,-}$ and $\W^{X,+}$.

\begin{lemma}\label{lemma.settles-W-does-not-solve-T}
  Suppose $f^\infty$ settles $\mathcal{W}^X_{e}$ for every $e$.  Then $f^\infty$ does not compute any solution to $T^\infty$.
\end{lemma}
\begin{proof}
  We must show that for each $e$, $W^{X\oplus f^\infty}_e$ is not a solution to $T^\infty$.  Since $f^\infty$ settles $\mathcal{W}^{X}_e$, there must be some $(x,z)$ demonstrating this.  
	
	If there is any $f\subset f^\infty$ with $f\in\mathcal{W}^{X,+}_e(x,z)$ then there is some $a\in W_e^{X\oplus f}\cap(x,z]$ and some $c\in W^{X\oplus f}_e\cap Y_a(T^\infty)$ with $c>z$.  Therefore $\{a,c\}\subseteq W^{X\oplus f^\infty}_e$ but there is no infinite path $\Lambda$ through $T^\infty$ with $\Lambda(a)=\Lambda(c)$, so $W^{X\oplus f^\infty}_e$ does not solve $T^\infty$.

 Therefore, suppose that there is no $f\subset f^\infty$ in $\mathcal{W}^{X,+}_e(x,z)$, and therefore no $f\subset f^\infty$ in $\mathcal{W}^{X,-}_e(x,z)$.  If there is no $f\subset f^\infty$ in $\mathcal{W}^{X,-,0}(x)$ then $W^{X\oplus f}_e\subseteq [0,x]$ for each $f\subset f^\infty$, so $W^{X\oplus f^\infty}_e$ is finite.

Otherwise there is some $f\subset f^\infty$ in $\mathcal{W}^{X,-,0}(x)$.  Then $\z^X_{\mathcal{W}_e}(f)\leq z$ (otherwise we would have $f\in\mathcal{W}^{X,-}_e(x,z)$).  If there were any $f'$ with $f\subseteq f'\subset f^\infty$ such that $W^{X\oplus f'}_e$ contains an element $>z$ then we would have $f'\in\mathcal{W}^{X,-}_e(x,z)$, so there is no such $f'$.  Therefore $W^{X\oplus f^\infty}_e\subseteq [0,z]$, and is therefore finite.
\end{proof}

\begin{definition}\label{defn.RKL-essential-and-dense}
$\mathcal{K}^X$ is \emph{essential} below $(f,n_\Xi,\Xi)$ if for every $x$ both (1) there is a $6n_\Xi-2$-branching set of extensions of $f$ contained in $\mathcal{K}^{X,-,0}(x)$, and (2) letting $U$ be the first such set enumerated, for every $z\geq \max_{f'\in U}\z^X_{\K}(f')$ there is a $4n_\Xi-1$-branching $U'\subseteq U$ such that for each $g\in U'$, there is a $4n_\Xi-1$-branching set of extensions of $g$, $U_g\subseteq\mathcal{K}^{X,-,1}(x,z)$.

$\mathcal{K}^X$ is \emph{uniformly dense} below $(f,n_\Xi,\Xi)$ if either $\mathcal{K}^X$ is not essential below $(f,n_\Xi,\Xi)$ or there is an $x$ a $2n_\Xi$-branching set of extensions $U$ of $f$, and a $z \geq \max_{f'\in U}\z^X_{\K}(f')$, such that $U\subseteq\mathcal{K}^{X,+}(x,z)$.  $\mathcal{K}^X$ is uniformly dense if it is uniformly dense below every $(f,n_\Xi,\Xi)$.
\end{definition}
The quantifier alternation in the definition of essential requirements reflects the need to resolve the two parts of the requirement simultaneously, without allowing new $\Pi_1$ commitments on $T^\infty$ to show up.  (Compare to the approach used to separate $\ADS$ from $\CAC$ in \cite{LST}.)

\begin{lemma}\label{lemma.RKL-uniform-dense-meets-or-avoids}
  Suppose $\mathcal{K}^X$ is uniformly dense.  Then for every $(f,n_\Xi,\Xi)$, there is an $(f',n'_\Xi,\Xi')\preceq(f,n_\Xi,\Xi)$ and an $(x,z)$ such that either $f'\in\mathcal{K}^{X,+}(x,z)$ or $\Xi'\cap\mathcal{K}^{X,-}(x,z)=\emptyset$.
\end{lemma}
\begin{proof}
First, suppose that for some $x$, whenever $U$ is an $6n_\Xi-2$-branching set of extensions of $f$, there is a $g\in U$ with $g\not\in\mathcal{K}^{X,-,0}(x)$.  
Let $\Xi'\subseteq\Xi$ consist of those $g$ such that $g\in\Xi\setminus\mathcal{K}^{X,-}(x)$. 
We claim that $\Xi'$ is blocking at width $6n_\Xi-2+n_\Xi-1$, so $(f,7n_\Xi-3,\Xi')\preceq(f,n_\Xi,\Xi)$ satisfies the second case.
To see this, recall by Lemma \ref{tree_pigeonhole} that any $6n_\Xi-2+n_\Xi-1$-branching set $U'$ of strings either has a $6n_\Xi-2$-branching subset of $\K^{X,-,0}(x)$ or a $n_\Xi$-branching subset of $U'\backslash \K^{X,-,0}(x)$.  
Because no $6n_\Xi-2$-branching set can be contained in $\K^{X,-,0}(x)$, it follows that there is an $n_\Xi$-branching subset $U''$ of $U'$ contained in $U'\backslash \K^{X,-,0}(x)$. 
Furthermore, because $\Xi$ is $n_\Xi$-branching, there is some $g\in\Xi\cap U''\subseteq \Xi\backslash\K^{X,-,0}(x)=\Xi'$, as desired.

Otherwise, for every $x$ we may find a $6n_\Xi-2$-branching set $U_x$ of extensions of $f$ with $U_{x}\subseteq\mathcal{K}^{X,-,0}(x)$.  For each $g\in U_{x}$ and each $z\geq \max_{f\in U}\z^X_{\K}(f)$, we ask whether there is a $4n_\Xi-1$-branching set of extensions $U_{g,z}\subseteq\mathcal{K}^{X,-,1}(x,z)$.  Let $U^0_{x,z}\subseteq U_x$ be the set of $g\in U_{x}$ such that there exists such a $U_{g,z}$.

Suppose that for some $x$ and some $z\geq \max_{f\in U}\z^X_{\K}(f)$, that $U^0_{x,z}$ does not contain a $4n_\Xi-1$-branching set of extensions.  Then $U_{x}\setminus U^0_{x,z}$ does contain a $2n_\Xi$-branching set of extensions.  By Lemma \ref{thm:dnr_extension}, there is a $g\in (U_{x}\setminus U^0_{x,z})\cap\Xi$ such that $(g,n_\Xi,\Xi\upharpoonright g)\preceq(f,n_\Xi,\Xi)$ is a condition.  Let $\Xi'\subseteq\Xi\upharpoonright g$ consist of those $g'\not\in\mathcal{K}^{X,-,1}(x,z)$.  Since there is no $4n_\Xi-1$-branching $U_{g,z}\subseteq \mathcal{K}^{X,-,1}(x,z)$ and because $z\geq\z^X_{\K}(g)$, $(g,5n_\Xi-2,\Xi')\preceq(f,n_\Xi,\Xi)$ is a condition and $\Xi'\cap\mathcal{K}^{X,-}(x,z)=\emptyset$.

Otherwise the sets $U_{x}$ witness that $\mathcal{K}^X$ is essential below $(f,n_\Xi,\Xi)$.
By the definition of uniform density, there is a $2n_\Xi$-branching set of extensions of $f$, $U\subseteq\mathcal{K}^{X,+}(x,z)$, and therefore a $(g,n_\Xi,\Xi\upharpoonright g)\preceq(f,n_\Xi,\Xi)$ such that $g\in\mathcal{K}^{X,+}(x,z)$.
\end{proof}

\begin{theorem}\label{thm:settling_everything_RKL}
  Fix any countable collection of uniformly dense requirements $\K^X_k$.  Then there is a $\mathrm{DNR}^X$ function $f^\infty$ settling every $\K^X_k$.
\end{theorem}
\begin{proof}
  Enumerate the uniformly dense requirements $\K^X_0,\K^X_1,\ldots$.  Construct a sequence of conditions $(f^0,n^0_\Xi,\Xi^0)\preceq(f^1,n^1_\Xi,\Xi^1)\preceq\cdots$ by repeatedly applying the previous lemma so that for each $k$, there is an $x$ such that either $f^k\in\mathcal{K}^{X,+}_k(x)$ or $\Xi^k\cap\mathcal{K}^{X,-}_k(x)=\emptyset$.  Then letting $f^\infty=\bigcup_k f^k$, $f^\infty$ is $\mathrm{DNR}^X$ by the definition of $\mathbb{P}^X$ and settles each $\mathcal{K}^X_k$.
\end{proof}

In summary, we have seen that settling each $\W^X$ ensures we avoid solving the instance $T^\infty$ of $\RKL$, and that we can settle each uniformly dense $\K^X$.
In the ground forcing construction, we will build $T^\infty$ so that all the requirements $\K^{T^\infty}$ are uniformly dense.

We now define the iteration requirements.
When satisfied, these requirements will ensure that other requirements remain uniformly dense with respect to $X\oplus f^\infty$.
This guarantees that the construction can be iterated to form a Turing ideal.

As before, the focus will be on moving from $\R^{X,-,1}(x,z)$ to $\R^{X,+}(x,z)$ (satisfying essential requirements), not on moving from $\R^{X,-,0}$ to $\R^{X,-,1}$ (determining if a requirement is essential).

\begin{definition}
  Let $e_0,e_1,e_2,e_3$ be given and let $q=(g,m_\Pi,\Pi)$ be a triple.  We let $\vec e$ be the sequence $e_0,e_1,e_2,e_3$.  
	
	Then $\mathcal{R}^{X,-,0}_{\vec e,q}(x)$ is the set of $f$ such that either:
  \begin{itemize}
  \item $f$ forces that $\mathcal{K}^{X\oplus f^\infty}_{\vec e}$ is not a $\mathrm{RKL}^{T^\infty}$-requirement, or
  \item $f$ forces that $q$ is not a condition, or
  \item There is a $6m_\Pi-2$-branching set of extensions of $g$ contained in $\mathcal{K}^{X\oplus f,-,0}_{\vec e}(x)$.
  \end{itemize}
If $f\in\mathcal{R}^{X,-,0}_{\vec e,q}(x)$ due to either of the first two cases, $\z^X_{\mathcal{R}_{\vec e,q}}(f)=0$.

When $f\in\mathcal{R}^{X,-,0}_{\vec e,q}(x)$ in the third case, we let $U^X_{\mathcal{R}_{\vec e,q}}(f,x)$ be the first $6m_\Pi-2$-branching set of extensions of $g$ contained in $\mathcal{K}^{X\oplus f,-,0}_{\vec e}(x)$ found.  In this case we set:
\[\z^X_{\mathcal{R}_{\vec e,q}}(f)=\max_{g'\in U^X_{\mathcal{R}_{\vec e,q}}(f,x)}\z^{X\oplus f}_{\mathcal{K}_{\vec e}}(g').\]

$\mathcal{R}^{X,-,1}_{\vec e,q}(x,z)$ is the set of $f$ such that either:
  \begin{itemize}
  \item $f$ forces that $\mathcal{K}^{X\oplus f^\infty}_{\vec e}$ is not a $\mathrm{RKL}^{T^\infty}$-requirement, or
  \item $f$ forces that $q$ is not a condition, or
  \item There is a $4m_\Pi-1$-branching $U'\subseteq U^X_{\mathcal{R}_{\vec e,q}}(f,x)$ such that for each $g'\in U'$, there is a $4m_\Pi-1$-branching set of extensions $U'_{g'}$ of $g'$ with $U'_{g'}\subseteq \mathcal{K}^{X\oplus f,-,1}_{\vec e}(x,z)$.
  \end{itemize}
$\mathcal{R}^{X,+}_{\vec e,q}(x,z,\{Y_a\}_{a\in(x,z]})$ is the set of $f$ such that either
  \begin{itemize}
  \item $f$ forces that $\mathcal{K}^{X\oplus f^\infty}_{\vec e}$ is not a $\mathrm{RKL}^{T^\infty}$-requirement, or
  \item $f$ forces that $q$ is not a condition, or
  \item There is a $2m_\Pi$-branching set of extensions of $g$, $$U'\subseteq \mathcal{K}_{\vec e}^{X\oplus f,+}(x,z,\{Y_a\}_{a\in(x,z]}).$$
  \end{itemize}
\end{definition}

\begin{lemma}
  $\mathcal{R}^X_{\vec e,q}$ is a $\mathrm{RKL}^{T^\infty}$-requirement.
\end{lemma}
\begin{proof}
We wish to show that $\R$ satisfies Definition \ref{defn.RKL-Requirement}. 
The uniformity properties of $\mathrm{RKL}^{T^\infty}$-requirements, the first three conditions, ensure the uniformity of convergence of the various sets used to define $\R^{X}_{\vec e,q}$.
The failure of any of these properties to hold of $\R^{X}_{\vec e,q}$ will therefore imply the failure in the uniformity of the corresponding computations in $\mathcal{K}^{X\oplus f}_{\vec e}$, causing $f$ to force $\mathcal{K}^{X\oplus f^\infty}_{\vec e}$ not to be a $\mathrm{RKL}^{T^\infty}$-requirement.

For the final property, suppose that $f\in\mathcal{R}^{X,-,0}_{\vec e,q}(x)\cap\mathcal{R}^{X,-,1}_{\vec e,q}(x,z)$, that $z\geq \z^X_{\mathcal{R}_{\vec e,q}}$, and that for each $a\in(x,z]$, $Y_{a,0}\cup Y_{a,1}$ is a partition of $(z,y')$ for $y'$ large enough to ensure termination of all computations.
We wish to show that there is a $j\in\{0,1\}$ such that $f\in \mathcal{R}^{X,+}_{\vec e}(x,z,\{Y_{a,j}\})$.

Because $f\in \R^{X,-,1}$, there is a $4m_\Pi-1$-branching $U'\subseteq U^X_{\mathcal{R}_{\vec e,q}}(f)$ such that for each $g'\in U'$, there is a $4m_\Pi-1$-branching set of extensions $U'_{g'}$ of $g'$ contained in $\mathcal{K}^{X\oplus f,-,1}_{\vec e}(x,z)$.  

We now use the fact that $\K^X$ is itself a requirement, and thus satisfies (4) of Definition \ref{defn.RKL-Requirement}. 
By our choice of $U'$, for each $g'\in U'$ and each $g''\in U'_{g'}$, $g''\in \K^{X\oplus f,-,0}_{\vec e}(x)\cap\K^{X\oplus f,-,1}_{\vec e}(x,z)$.
Furthermore, by definition of $\R_{\vec e,q}$, $z\geq \z^X_{\mathcal{R}_{\vec e,q}}\geq \z^{X\oplus f}_{\mathcal{K}_{\vec e}}(g')=\z^{X\oplus f}_{\mathcal{K}_{\vec e}}(g'')$.
Because $\K^{X\oplus f}_{\vec e}$ is a requirement, it follows that for each $g'\in U'$, for each $g''\in U'_{g'}$ there is a $j_{g''}$ such that $g''\in\mathcal{K}^{X\oplus f,+}_{\vec e}(x,z,\{Y_{a,j_{g''}}\})$.  Applying Lemma \ref{tree_pigeonhole} to each $U'_{g'}$, for each $g'\in U'$ there is a $j_{g'}$ so that there is a $2m_\Pi$-branching subset of $U'_{g'}$ with $j_{g''}=j_{g'}$.  Applying Lemma \ref{tree_pigeonhole} to $U'$, there is a $j$ and a $2m_\Pi$-branching subset of $U'$ with $j_{g'}=j$.  

Examining the definition of $\K^{X\oplus f,-,1}_{\vec e}(x,z)$ and $\mathcal{K}^{X\oplus f,+}_{\vec e}(x,z,\{Y_{a,j_{g''}}\})$ for the relevant $g''$ extending $g$, we see that $j$ witnesses that $f\in\mathcal{R}^{X,+}_{\vec e,q}(x,z,\{Y_{a,j}\})$.
\end{proof}

The definition of $\R$ is modeled after the definition of uniform density for $\K$.  
Unpacking this yields the following lemma.

\begin{lemma}
  If $f^\infty$ settles $\mathcal{R}^X_{\vec e,q}$ for every $q$ then $\K^{X\oplus f^\infty}_{\vec e}$ is uniformly dense.
\end{lemma}
\begin{proof}
  Let $q$ be a condition and suppose $\mathcal{K}^{X\oplus f^\infty}_{\vec e}$ is essential below $q=(g,m_\Pi,\Pi)$.  Let $(x,z)$ be the value for which $f^\infty$ settles $\mathcal{R}^X_{\vec e,q}$.

Let $U\subseteq\mathcal{K}^{X\oplus f^\infty,-,0}_{\vec e} (x)$ be the $6m_\Pi-2$-branching set of extensions of $g$ ensured by essentialness.  Then $U=U^X_{\vec e,q}(f,x)$ for any $f\subset f^\infty$ large enough to make the necessary computations converge.  There is a $4m_\Pi-1$-branching $U'\subseteq U$ such that for each $g\in U'$, there is a $4m_\Pi-1$-branching set of extensions of $g$, $U_g\subseteq\mathcal{K}^{X\oplus f^\infty,-,1}_{\vec e} (x,z)$.  Therefore there is an $f\subset f^\infty$ such that $f\in\mathcal{R}^{X,-}_{\vec e,q}(x,z)$.

Since $f^\infty$ settles $\mathcal{R}^X_{\vec e,q}$, there is some $f\subset f^\infty$ with $f\in\mathcal{R}^{X,+}_{\vec e,q}(x,z)=\mathcal{R}^{X,+}_{\vec e,q}(x,z,\{Y_a(T^\infty)\})$.  Therefore there is a $2m_\Pi$-branching $U'\subseteq U$ so that for each $g'\in U'$, there is a $2m_\Pi$-branching set of extensions $U_{g'}$ of $g'$ consisting of $g''\in\mathcal{K}_{\vec e}^{X,+}(x,z,\{Y_a(T^\infty)\})=\mathcal{K}_{\vec e}^{X,+}(x,z)$.  Then $\bigcup_{g'\in U'}U_{g'}$ is a $2m_\Pi$-branching set of extensions of $g$ contained in $\mathcal{K}_{\vec e}^{X,+}(x,z)$.
\end{proof}

We conclude:
\begin{theorem}\label{thm:rkl_iterate}
  Suppose $T^\infty\leq_T X$ are such that every requirement $\mathcal{K}^X_{\vec e}$ is uniformly dense.  Then there is a $\mathrm{DNR}^X$ function $f^\infty$ such that no $\Phi^{X\oplus f^\infty}_e$ is a solution to the \RKL{} instance $T^\infty$ and every requirement $\mathcal{K}^{X\oplus f^\infty}_{\vec e}$ is uniformly dense.
\end{theorem}

\subsection{Ground Forcing}

We now describe a forcing notion for building instances of \RKL.  Recall the symmetry notion on trees above: $\sigma_{/x}$ is the sequence with $\sigma_{/x}\upharpoonright[0,x]=\sigma\upharpoonright[0,x]$ and $\sigma_{/x}\upharpoonright(x,|\sigma|)=1-(\sigma\upharpoonright(x,|\sigma|))$, and a tree $T$ is symmetric across $S$ if whenever $\sigma\in T$ and $x\in S$, $\sigma_{/x}\in T$.

We force with pairs $(T,S)$ where $S$ is a finite set and $T$ is a finite $\{0,1\}$-tree symmetric across $S$.

We say $(T',S')\preceq(T,S)$ if $T'$ extends $T$ and $S\subseteq S'$.  We say $T'\preceq(T,S)$ if $T'$ extends $T$ and is symmetric across $S$.

\begin{lemma}\label{lemma.groundforcing}
For any $(T,S)$, any $\vec e=e_0,e_1,e_2,e_3$, and any $p=(f,n_\Xi,\Xi)$, there is a $(T',S')\preceq(T,S)$ such that either:
\begin{itemize}
\item $(T',S')$ forces that $\mathcal{K}^{T^\infty}_{\vec e}$ is not an $\mathrm{RKL}^{T^\infty}$-requirement,
\item $(T',S')$ forces that $p$ is not a condition, or
\item $\mathcal{K}^{T^\infty}_{\vec e}$ is uniformly dense below $p$.
\end{itemize}
\end{lemma}
\begin{proof}
If any $(T',S')\preceq(T,S)$ forces either of the first two cases, we are finished, so assume not.  Let $x$ be the smallest number larger than $|T|$ and any element in $S$.  We ask:
\begin{quote}
Is there any $T'\preceq(T,S\cup\{x\})$ such that there is a $6n_\Xi-2$-branching set of extensions $U$ of $f$ in $\mathcal{K}_{\vec e}^{T',-,0}(x)$?
\end{quote}

Suppose not.  Then $(T,S\cup\{x\})$ forces that $\mathcal{K}^{T^\infty}_{\vec e}$ is not essential.

Suppose so, and let $U$ be the first such set enumerated.  Without loss of generality, we may assume that $|T'|\geq\max_{g\in U}\z^{T'}_{\mathcal{K}_{\vec e}}(g)$ for each $g\in U$.  Choosing $z$ to be larger than $x$ and $|T'|$, we ask:
\begin{quote}
  Is there any $T''\preceq(T',S\cup\{x,z\})$ such that there is a $4n_\Xi-1$-branching $U'\subseteq U$ such that for each $g\in U'$, there is a $4n_\Xi-1$-branching set of extensions of $g$, $U_g\subseteq\mathcal{K}^{T'',-,1}_{\vec e}(x,z)$?
\end{quote}

Suppose not, and recall that $z> |T'|\geq\max_{g\in U}\z^{T'}_{\mathcal{K}_{\vec e}}(g)$.
Thus, by Definition \ref{defn.RKL-essential-and-dense}, we see $(T',S\cup\{x,z\})$ forces that $\mathcal{K}^{T^\infty}_{\vec e}$ is not essential.

Suppose so, and let $y'$ be large enough to witness the convergence of all necessary computations.  Without loss of generality, we may assume that $|T''|\geq y'$.  Then $U''=\bigcup_{g\in U'}U_g$ is a $4n_\Xi$-branching set of extensions of $g$.  Because property (4) of the definition of a requirement satisfies the assumptions of Lemma \ref{thm:sym_breaking}, we may apply it to obtain a $T^*\preceq (T'',S\cup\{x\})$ and a $2n_\Xi$-branching $U^*$ so that for each $\sigma\in T^*$ with $|\sigma|=|T^*|$, $U^*\subseteq\mathcal{K}^{T^*,+}_{\vec e}(x,z,\{Y_a(\sigma)\})$.  This shows that $\mathcal{K}^{T^*}_{\vec e}$ is uniformly dense below $p$.
\end{proof}
In this proof, the symmetry over $x$ is not really important.  We need symmetry in extensions above the elements of $S$, so it is convenient to fix a value $x$ and assume symmetry over $x$.  The important point is symmetry over $z$---we assume this symmetry, and if we find the desired sets $U_g$, we ``spend,'' or ``sacrifice,'' this symmetry to force uniform density.  

In other words, $T''$ contains the four possible configurations $\tau_0{}^\frown\tau_1$, $\tau_0{}^\frown(1-\tau_0)$, $(1-\tau_0){}^\frown\tau_1$, and $(1-\tau_0){}^\frown(1-\tau_1)$.  When we extend to $T^*$, we kill off two of these configurations and the symmetry over $z$ while retaining the other two, and therefore retaining symmetry above $x$ (and also above the values in $S$, which are all $\leq x$).

\begin{theorem}\label{thm:rkl_main}
There is a Turing ideal $\mathcal{I}$ satisfying $\DNR$ but not $\RKL$.
\end{theorem}
\begin{proof}
  Applying the previous lemma repeatedly, we may construct a $T^\infty$ so that every requirement $\mathcal{K}^{T^\infty}$ is uniformly dense.  By Theorem \ref{thm:rkl_iterate} we may find a $\mathrm{DNR}^{T^\infty}$ function $f^\infty$ so that no $\Phi^{T^\infty\oplus f^\infty}_e$ is  a solution to the \RKL instance $T^\infty$ and every requirement $\mathcal{K}^{T^\infty\oplus f^\infty}$ is uniformly dense.   Iterating this countably many times and closing under computability, we obtain a Turing ideal $\mathcal{I}$ containing $T^\infty$, containing no solution to the \RKL{} instance $T^\infty$, and so that for every $W\in\mathcal{I}$, $\mathcal{I}$ contains a $\mathrm{DNR}^W$ function.
\end{proof}

\begin{cor}
\DNR{} does not imply \RKL.
\end{cor}

\subsection{Priority Construction for \texorpdfstring{$T^\infty$}{T}}

Above and in \cite{LST} the construction of the difficult instance of a problem is given by a forcing arugment.  As noted in \cite{LST}, this forcing can be replaced by a priority argument.  The argument is slightly more complicated, but has the benefit that it shows that the instance of \RKL{} can be taken to be computable.  For completeness (and to illustrate how the conversion can be done for all arguments of this kind) we illustrate this with the construction of $T^\infty$.

We take the countably many values of $\vec e$ and $p$ and arrange them in a priority list of order type $\omega$; we call the data associated to the $i$-th requirement in this list $(\mathcal{K}^{T^\infty}_{\vec e_i},f_i,n_i,\Xi_i)$.  We construct the tree $T^\infty$ in infinitely many stages; at the $n$-th stage, we have a finite tree $T_n$.  We also maintain, at each stage, a set $S_n$.

At each stage, a requirement can be in one of four states, which we call $I,0,1,+$; initially all requirements are in the $I$ stage.  A requirement can move from $I$ to $0$, from $0$ to $1$, and from $1$ to $+$; when a requirement changes states it may injure all lower priority requirements, reseting their state to $I$.

We will define below what it means for a requirement to need attention at stage $n$, and what we do when a requirement needs attention.  Given $T_n$ and the associated date, we define $T_{n+1}$ by finding the least $i\leq n$ which needs attention and apply the corresponding operations.

A requirement $i$ in state $I$ always needs attention.  When it receives attention at stage $n$ we define a value $x_i$ to be larger than any value in $S_n$ and let $S_{n+1}=S_{n}\cup\{x_i\}$.  We move this requirement to state $0$.  All other requirements are untouched.  We define $T_{n+1}$ to consist of all $\sigma^\frown\langle b\rangle$ with $\sigma\in T_n$, $|\sigma|=|T_n|$, and $b\in\{0,1\}$.  (All lower priority requirements are in state $I$, so we do not need to worry about injuring them.)

A requirement in state $0$ needs attention at stage $n$ if there exists a $6n_i-2$-branching set of extensions $U_i$ of $f_i$ so that for each $g\in U$, $g$ is coded by a value $\leq n$ and there is a $y\leq n$ so that $\exists g'\subseteq g R_{e_0}^{T_n,-,0}(x_i,y,g')$.  (That is, $U\subseteq\mathcal{K}^{T_n,-,0}_{\vec e_i}(x_i)$ with all necessary witnesses bounded by $n$.)  We take the first such $U$.  We choose a value $z_i\geq\bigcup_{g\in U}\z^{T_n}_{\mathcal{K}_{\vec e_i}}(g)\cup S_n$.  We let $S_{n+1}=S_n\cap[0,x_i]\cup\{z_i\}$.  We injure all lower priority requirements.  (Any requirement which had $x_{i'}>x_i$ or $z_{i'}>x_i$ would have to be injured; our construction ensures that all such requirements are indeed lower priority.)  We move this requirement to state $1$.

A requirement in state $1$ needs attention at stage $n$ if there exists a $4n_i-1$-branching $U'\subseteq U_i$ so that for each $g\in U'$ there is a $4n_i-1$-branching set of extensions $U_g$ of $g$ so that for each $g'\in U_g$, $g'$ is coded by a value $\leq n$ and there is a $y\leq n$ so that $\exists g''\subseteq g' R_{e_1}^{T_n,-,1}(x_i,z_i,y,g'')$.  As above, we apply Lemma \ref{thm:sym_breaking} to obtain an extension $T_{n+1}=T^*$ of $T_n$ so that there is a $2n_i$-branching $U^*$ so that for each $\sigma\in T_{n+1}$ with $|\sigma|=|T_{n+1}|$, $U^*\subseteq\mathcal{K}^{T_{n+1},+}_{\vec e}(x_i,z_i,\{Y_a(\sigma)\})$.  We set $S_{n+1}=S_n\cap[0,x_i]$; this injures all lower priority requirements.

Observe that this is a finite injury construction: a given requirement is only injured when a higher priorty requirement moves from $0$ to $1$ or from $1$ to $+$.  By an easy induction, each requirement eventually stabilizes in some stage and eventually receives attention if it needs attention.

We check that the result forces each $\mathcal{K}^{T^\infty}_{\vec e}$ to be uniformly dense.  If $\mathcal{K}^{T^\infty}_{\vec e_i}$ is essential below $p_i$ then consider the $x_i$ chosen by the requirement $i$ the last time it entered state $0$.  Being essential means $i$ needed attention at some large enough state $n$ after the last time it entered stage $0$, and therefore eventually received attention and moved to state $1$.  At this time some value $z_i$ was fixed, and being essential means $i$ needed attention again at some later stage.  At this stage $i$ was moved to state $+$, and the construction of $T_{n+1}$ in state $+$ ensures that $\mathcal{K}^{T^\infty}_{\vec e_i}$ is uniformly dense below $p_i$.  Since each pair $\vec e,p$ is $\vec e_i,p_i$ for some $i$, this completes the claim.

We conclude with a slight strengthening of Theorem \ref{thm:rkl_main}:
\begin{theorem}
  There is a computable tree $T^\infty$ and a Turing ideal $\mathcal{I}$ satisfying \DNR{} so that $\mathcal{I}$ contains no solution to $T^\infty$ as an \RKL{} instance.
\end{theorem}

\section{Positive Implications}\label{sec:positive}

\begin{definition}
A \emph{tournament} is a binary relation $\rightarrow$ on $S\subseteq\mathbb{N}$ such that for every pair $x\neq y$, exactly one of $x\rightarrow y$ and $x\leftarrow y$ holds.  $\rightarrow$ is \emph{transitive} if $x\rightarrow y$ and $y\rightarrow z$ implies $x\rightarrow z$.

  \EM{} states that whenever $\rightarrow$ is a tournament there is an infinite set $H$ such that $\rightarrow$ is transitive on $H$.

A tournament is \emph{stable} if for every $x$, either $x\rightarrow y$ for cofinitely many $y$, or $x\leftarrow y$ for cofinitely many $y$.  \SEM{} states that whenever $\rightarrow$ is a stable tournament there is an infinite set $H$ such that $\rightarrow$ is transitive on $H$.
\end{definition}

\begin{theorem}
  \SEM{} implies \RKL.
\end{theorem}
This was also shown independently in \cite{RKL-Variants}.
\begin{proof}
Our proof adapts the argument from \cite{RKL} that \SRT{} implies \RKL.  Let $T$ be an infinite tree of $\{0,1\}$ sequences---that is, an instance of \RKL.  For any $y$, let $\sigma_y$ be the lexicographically leftmost sequence in $T$ of length $y$.  Then for any pair $(x,y)$ with $x<y$, we set $x\leftarrow y$ if $\sigma_y(x)=0$ and $x\rightarrow y$ if $\sigma_y(x)=1$.  For any $x$, $\lim_y\sigma_y(x)$ exists, so the tournament is stable.

By \SEM, we have an infinite set $S$ on which $\rightarrow$ is transitive.  This has the following consequence: for any finite $S_0\subseteq S$, there are arbitrarily long $\sigma\in T$ such that $x\leftarrow x'$ with $x,x'\in S_0$ implies $\sigma(x)\leq\sigma(x')$.  In other words, viewing $\leftarrow$ as a linear ordering on $S$, for any finite subset $S_0$ we find long $\sigma$ so that $\sigma:S_0\rightarrow\{0,1\}$ is order preserving.

Let $T'\subseteq T$ consist of those $\sigma\in T$ such that $\sigma$ is order preserving on $S\cap[0,|\sigma|)$; $T'$ must also contain arbitrarily long sequences and is computable from $S$.

Let's define $\omega(S)$ to be those $x\in S$ such that there are only finitely many $y\in S$ with $y\leftarrow x$ (that is, those $x$ with finitely many predecessors).  Similarly, define $\omega^*(S)$ to be those $x\in $ so there are only finitely many $y\in S$ with $y\rightarrow x$.

Suppose that for some $x\in S\setminus \omega^*(S)$ we have $\sigma(x)=1$ for arbitrarily large $\sigma\in T'$.  Then let $S^*=\{y\in S\mid y\rightarrow x\}$.  Whenever $\sigma(x)=1$ for $\sigma\in T'$, also $\sigma(y)=1$ for all $y\in S^*$, so $S^*$ is a solution to $T'$ as an instance of \RKL, so also to $T$.  Similarly, if there is an $x\in S\setminus \omega(S)$ so that $\sigma(x)=0$ for arbitrarily large $\sigma\in T'$ we could similarly use $\{y\mid y\leftarrow x\}$.

So consider the remaining case.  Clearly $S=\omega(S)\cup\omega^*(S)$ and for any $x\in \omega(S)$, every sufficiently long $\sigma\in T'$ makes $\sigma(x)=0$ while for any $x\in\omega^*(S)$, every sufficiently long $\sigma\in T'$ makes $\sigma(x)=1$.  Then $\omega(S)$ and $\omega^*(S)$ are computable: given $x\in S$, there must be some $n$ and some $b\in\{0,1\}$ so that for every $\sigma\in T'$ with $|\sigma|=n$, $\sigma(x)=b$.  Then $x\in\omega(S)$ iff $b=0$.  At least one of $\omega(S),\omega^*(S)$ is infinite, and there are arbitrarily long $\sigma\in T'\subseteq T$ so $\sigma(x)=0$ for $x\in\omega(S)$ and $\sigma(x)=1$ for $x\in\omega^*(S)$.  Therefore whichever of these sets is infinite is a solution to $T$ as an instance of \RKL.
\end{proof}

The following result has been shown by Bienvenu, Patel, and Shafer \cite{RKL-Variants}.  We give an alternate direct proof:
\begin{theorem}
  \DNR{} implies \WRKL{}.
\end{theorem}
\begin{proof}
If $S\subseteq 2^{<\omega}$ is a set of finite sequences, we write $[S]\subseteq 2^{\omega}$ for the collection of infinite sequences with some initial segment in $S$.  

We are given a tree $T$ so that $[T]$ has measure $\geq\epsilon$.  We will exhibit explicit Turing functionals $\Psi^T$ and $\Upsilon^X$ such that if $f:\omega\rightarrow\omega$ is a $\DNR^{\Psi^T}$ function, $\Upsilon^f$ will compute a solution to $T$.

We describe $\Upsilon^X$ first, since it is quite explicit, and doesn't even depend on $T$.  The idea is that we view $\Upsilon^X$ as a map from $f:\omega\rightarrow\omega$ to an infinite sequence $n\mapsto \Upsilon^f(n)$ in such a way that it requires many values of $f$ to determine $\Upsilon^f(n)$.

We fix a computable function $r(i)$ which is sufficiently quickly growing.  (The exact value can be calculated from the work below.)  Let $p_0,p_1,\ldots$ be the sequence of primes.  If $f$ is a function with $[0,r(i)]\subseteq\dom(f)$, we define $a_i^f=\prod_{j\leq i}p_j^{f(j)}$ and $b_i^f=\prod_{j\leq r(i)}p_j^{a_j^f}$.  We set $\Upsilon^f(i)=b_i^f$.

The main useful feature of these sequences is that if $\vec s^0=s^0_0,\ldots,s^0_{r(i)}$ and $\vec s^1=s^1_0,\ldots,s^1_{r(i)}$ are sequences such that for some $i'\leq i$ and some $j\leq r(i')$, $s^0_j\neq s^1_j$, then $b_i^{\vec s^0}\neq b_i^{\vec s^1}$.  For a given finite sequence $\vec s=s_0,\ldots,s_{r(i)}$, it is natural to focus on the set of numbers which could show up in possible extensions---that is, the set $N^{\vec s}\subseteq\mathbb{N}$ such that $x\in N^{\vec s}$ iff there is some sequence $\vec s'\sqsupseteq\vec s$ so that $x=b^{\vec s'}_{i'}$ for some $i'$.  There is a natural function $\pi_{\vec s}:2^{\omega}\rightarrow 2^{N^{\vec s}}$ given by $\pi_{\vec s}(\Lambda)(i)=\Lambda(i)$ so that $\mu(\pi^{-1}_{\vec s}(S))=\mu(S)$ for any measurable set $S\subseteq 2^\omega$.  We will use measures of the form $\mu(\pi_{\vec s}(S))$ because this lets us focus on the contributions of difference choices of $\vec s$ without entangling their effects.

For any $\vec s=s_0,\ldots,s_{r(i)}$ and $n> b_i^{\vec s}$, let us write
\[\langle \vec s\rangle_n=\{\sigma\in 2^n\mid\forall i_0,i_1\leq i\ \sigma(b^{\vec s}_{i_0})=\sigma(b^{\vec s}_{i_1})\},\]
so $\mu([\langle \vec s\rangle_n])=2^{1-i}$.  Since $[\langle \vec s\rangle_n]=[\langle \vec s\rangle_m]$ for $n,m>b_i^{\vec s}$, we just write $[\langle \vec s\rangle]$ for this set.

We describe an enumeration algorithm for $\Psi^T$---that is, a single algorithm which may, at a given stage, output a value $\Psi^T(j)=s$ for some $j$ where a value has not already been set.  We will also construct a tree $T'\supseteq T$; at the $n$-th stage we will specify which sequences of length $n$ belong to $T'$.

We say a sequence $\vec s=s_0,\ldots,s_{r(i)}$ has been \emph{killed} by stage $n$ if $T\cap\langle \vec s\rangle_n=\emptyset$.  This means that $T$ has ruled out $b^{\vec s}$ as a possible beginning to an infinite branch.  

At the $n$-th stage, $\Psi^T$ \emph{takes notice of} the first $\vec s=s_0,\ldots,s_{r(i)}$ which has been killed by stage $n$ but which we have not taken notice of at a previous stage.  (The exact ordering used does not matter as long as all sequences eventually get considered; a simple choice is to order sequences first by the earliest stage at which they are killed, and then order sequences killed at the same stage lexicographically.)  A sequence of length $n$ belongs to $T'$ exactly if it extends a sequence of length $n-1$ in $T'$ and does not belong to $\langle \vec s\rangle_{n}$.

Let $i'\leq i$ be least such that
\[\frac{\mu(\pi_{s_0,\ldots,s_{r(i')}}([T'_{n}\cap \langle s_0,\ldots,s_{r(i')}\rangle_n]))}{\mu([\langle s_0,\ldots,s_{r(i')}\rangle_n])}<\epsilon/2^{i'}.\]
Such an $i'$ exists since $T'_{n}\cap\langle \vec s\rangle_n=\emptyset$.  If there is any $j\in(r(i'-1),r(i')]$ such that we have not yet output a value at $j$, we output $\Psi^T(j)=s_j$ for the least such $j$; otherwise we output nothing at this stage.  This completes the definition of the algorithm $\Psi^T$.

We now argue that if there is an $i$ so that $\Psi^T(j)$ is defined for every $j\in(r(i-1),r(i)]$, then $\mu([T'])<\epsilon$.  Suppose $\Psi^T(j)$ is defined for every $j\in(r(i-1),r(i)]$; then for each such $j$ we found a sequence $\vec s^j$ of length $r(i)+1$ so that $\frac{\mu(\pi_{\vec s^j}([T']\cap[\langle\vec s^j\rangle]))}{\mu([\langle\vec s^j\rangle])}<\epsilon/2^i$.  We define a coloring $c$ on distinct pairs $j,j'$ by setting $c(j,j')<i$ to be the largest $i'$ such that $s^j_0,\ldots,s^j_{r(i'-1)}=s^{j'}_0,\ldots,s^{j'}_{r(i'-1)}$, or $0$ if there is no such $i'$.  By the finite Ramsey's Theorem there is an $i'<i$ and a subset $H\subseteq (r(i-1),r(i)]$ with 
\[|H|> \frac{\log_2 \epsilon - i'}{\log_2[1-2^{i'-i}+\epsilon 2^{i'-2i}]}\]
 so that for all distinct pairs $j,j'\in H$, $c(j,j')=i'$.  ($r(i)$ should be chosen large enough to ensure this instance of the finite Ramsey's Theorem.)

Let $\vec s^*$ be the common initial segment of length $r(i'-1)+1$ (which is empty if $i'=0$).  For any $\vec s^j,\vec s^{j'}$ with $j,j'\in H$, $b^{\vec s^j}_{i''}=b^{\vec s^{j'}}_{i''}$ iff $i''\leq i'$.  Therefore the sets $[\langle \vec s^j\rangle]$ are all subsets of $[\langle \vec s^*\rangle]$ and independent as subsets.  Observe that $\frac{\mu([\langle \vec s^j\rangle])}{\mu([\langle \vec s^*\rangle])}=2^{i'-i}$.  If $\Lambda\in[T']\cap[\langle \vec s^*\rangle]$ then for each $j$ we have either $\Lambda\in [\langle \vec s^j\rangle]$ or $\Lambda\in [\langle \vec s^*\rangle]\setminus[\langle \vec s^j\rangle]$.  Observe that $\frac{\mu_{\vec s^j}([T']\cap[\langle \vec s^j\rangle])}{\mu([\langle \vec s^j\rangle])}<\epsilon 2^{-i}$, so $\frac{\mu_{\vec s^j}([T']\cap[\langle \vec s^j\rangle])}{\mu([\langle \vec s^*\rangle])}<\epsilon 2^{i'-2i}$.  On the other hand $\mu([\langle \vec s^*\rangle]\setminus[\langle \vec s^j\rangle])=1-2^{i'-i}$.

Since these sets are independent as subsets of $[\langle \vec s^*\rangle]$, it follows that $\mu_{\vec s^*}([T']\cap[\langle s^*\rangle])<(1-2^{i'-i}+\epsilon 2^{i'-2i})^{|H|}$.  We have chosen $H$ large enough that this is $<\epsilon 2^{-i'}$.  Consider the largest $j\in H$ and the stage at which $\Psi^T(j)$ was output.  At this stage, we picked the value $\vec s^j$, and we have seen that
\[\frac{\mu(\pi_{s^j_0,\ldots,s^j_{r(i'-1)}}([T'\cap \langle s^j_0,\ldots,s^j_{r(i'-1)}\rangle]))}{\mu([\langle s^j_0,\ldots,s^j_{r(i'-1)}\rangle])}<\epsilon/2^{i'}.\]
If $i'>0$ then we could only have output some $\Psi^T(j')$ with $j'\leq r(i'-1)$, and therefore we could not have output $\Psi^T(j)$.  This is a contradiction, so we cannot have $i'>0$.  But then $i'=0$, so $\mu([T'])<\epsilon$, and since $T\subseteq T'$, we would have $\mu([T])<\epsilon$.

So if $\mu([T])\geq\epsilon$, for each $i$ there is a $j\in(r(i-1),r(i)]$ so that $\Phi^T_e(j)\uparrow$.  Now, we show that if $f$ is a total function such that whenever $\Phi^T_e(n)\downarrow$, $f(n)\neq \Phi^T_e$, then $\Phi^f_m$ is a solution to $T$.  For suppose not; then there is some $i$ and some level $n$ of $T$ such that $b^f_0,\ldots, b^f_i$ is not a path through any $\sigma\in T_n$.  This means that some initial segment $f(0),\ldots,f(i)$ was killed, and therefore we took notice of it at some stage $i'\leq i$ and tried to output $\Phi^T_e(j)=f(j)$ for some $j\in (r(i'-1),r(i)]$; since this could not have occured, it must be that $\mu(T)<\epsilon$, contradicting the assumption that $\mu(T)\geq \epsilon$.
\end{proof}

\section{Further Questions}

There is still some room to refine the results here and in \cite{RKL-Variants}.  Patey \cite{Patey} shows that $\RKL$ (and even stronger theories like $\mathbf{RT}^2_2+\mathbf{FS}$) do not imply $d\mhyphen\DNR$ for any $d$.  The reverse implication still seems to be open, however.

\begin{question}$\phantom{.}$
For which $d$ does $d\mhyphen\DNR$ imply $\RKL$?
\end{question}

\bibliographystyle{plain}
\bibliography{LST}
\end{document}